\documentclass[11pt,reqno,tbtags]{amsart}
\usepackage{amssymb}
\numberwithin{equation}{section}

\makeatletter
\newtheorem*{rep@theorem}{\rep@title}
\newcommand{\newreptheorem}[2]{%
\newenvironment{rep#1}[1]{%
 \def\rep@title{#2 \ref{##1}}%
 \begin{rep@theorem}}%
 {\end{rep@theorem}}}
\makeatother

\newtheorem{theorem}{Theorem}[section]
\newreptheorem{theorem}{Theorem}
\newtheorem{lemma}[theorem]{Lemma}

\newtheorem{corollary}[theorem]{Corollary}

\theoremstyle{definition}

\theoremstyle{remark}

\newcounter{thmenumerate}

\newcounter{xenumerate}

\newcommand\bbR{\mathbb R}

\newcommand\bbN{\mathbb N}

\newcommand\bbZ{\mathbb Z}

\newcommand\E{\operatorname{\mathbb E{}}}
\renewcommand\P{\operatorname{\mathbb P{}}}

\newcommand\eps{\varepsilon}

\renewcommand\phi{\varphi}

\newcommand\cF{\mathcal F}

\newcommand\N{{\mathbb N}}
\newcommand\Z{{\mathbb Z}}

\newcommand\R{{\mathbb R}}

\newcommand\1{{\mathbb I}}

\begin{document}
\title%[]
{Vertices of high degree in the preferential attachment tree}

%\date{19 April 2010} % (typeset \today)}

\author{Graham Brightwell}
\address{Department of Mathematics, London School of Economics,
Houghton Street, London WC2A 2AE, United Kingdom}
\email{g.r.brightwell@lse.ac.uk}
\urladdr{http://www.maths.lse.ac.uk/Personal/graham/}

\author{Malwina Luczak}
\thanks{Malwina Luczak's research is partially supported by an EPSRC Leadership Fellowship}
\address{School of Mathematics and Statistics, University of Sheffield,
  Hicks Building, Hounsfield Rd, Sheffield S3 7RH, United Kingdom}
\email{m.luczak@sheffield.ac.uk}

%{60C05 (68P10,68W40)} %%{Primary: <subject>; Secondary: <subject>}

\begin{abstract}
We study the basic preferential attachment process, which generates a sequence of random trees, each obtained from the
previous one by introducing a new vertex and joining it to one existing vertex, chosen with probability proportional to its degree.
We investigate the number $D_t(\ell)$ of vertices of each degree $\ell$ at each time $t$, focussing particularly on
the case where $\ell$ is a growing function of $t$.  We show that $D_t(\ell)$ is concentrated around its mean, which is
approximately $4t/\ell^3$, for all $\ell \le (t/\log t)^{-1/3}$; this is best possible up to a logarithmic factor.
%We prove similar results for the number $U_t(\ell)$ of vertices of degree at least $\ell$ at each time $t$.
\end{abstract}

\maketitle

\section{Introduction}\label{Sintro}

In this paper, we study the basic {\em preferential attachment} process, which is defined as follows.
We start with a (small) tree on $\tau_0 \ge 1$ vertices.  At each integer time $t > \tau_0$, a new vertex arrives,
and is joined to one existing vertex; a vertex is chosen as the other endpoint of the new edge with probability proportional
to its current degree.  Thus, at each time $t$, we have a tree with $t$ vertices.  The random tree obtained at any time
$t$ is called the {\em preferential attachment tree}.

\smallskip

The first appearance of this process can be traced back at least to Yule~\cite{yule} in 1925, and in probability theory
the model is sometimes referred to as a {\em Yule process}.  Subsequently, Szyma\'nski~\cite{Szy1} studied the
preferential attachment process in the guise of
{\em plane-oriented recursive trees}.  He gave a formula for the expected number $d_t(\ell)$ of vertices of
degree $\ell$ at time $t$, namely
$$
d_t(\ell) = \frac{4t}{\ell(\ell+1)(\ell+2)} + O(1).
$$
The structure of such trees was further analysed by Mahmood, Smythe and
Szyma\'nski~\cite{MSS}, and by Mahmood and Smythe~\cite{MS}.  Lu and Feng~\cite{LF} proved a concentration result for
the random number $D_t(\ell)$ of vertices of degree $\ell$, for fixed $\ell$.

Interest in the model surged after a paper of Barab\'asi and Albert~\cite{BA} in 1999, who proposed preferential
attachment as a model of the growth of ``web graphs'', i.e., graphs possessing many of the same properties as
``real-world networks'' such as the worldwide web.
Barab\'asi and Albert studied not just the preferential attachment process as defined above, but also the variant where each
new vertex chooses some fixed number $m \ge 1$ of neighbours.  For $m > 1$, the {\em preferential attachment graphs} produced
are of course not trees, but for properties such as the degree sequence, the overall pattern of behaviour is the same for any fixed $m$.

Preferential attachment graphs were studied formally by Bollob\'as, Riordan, Spencer and Tusn\'ady~\cite{brst01},
who proved that the degree sequence follows a power law with exponent~3, i.e., the expected number $d^m_t(\ell)$ of
vertices of degree $\ell$ at time $t$ is of order $t/\ell^3$, more precisely
$$
d^m_t(\ell) \simeq \frac{2m(m+1)}{(\ell+m)(\ell+m+1)(\ell+m+2)} \, t
$$
for all $\ell \le t^{1/15}$.  They also showed that the random number $D^m_t(\ell)$ of vertices
of degree $\ell$ at time $t$ is concentrated within $O(\sqrt {t \log t})$ of its expectation $d^m_t(\ell)$.
They further indicated how the results could be extended to somewhat larger values of $\ell$.

Szyma\'nski~\cite{Szy2} gave a more precise estimate for $d^1_t(\ell) = d_t(\ell)$.  Combining this with the concentration
result of Bollob\'as, Riordan, Spencer and Tusn\'ady~\cite{brst01} shows that $D_t(\ell)$ is concentrated within
a factor $(1+o(1))$ of its mean for $\ell$ up to nearly $t^{1/6}$.
Szyma\'nski~\cite{Szy2} also gave a precise estimate for the expected number $u_t(\ell)$ of vertices of degree {\em at least}
$\ell$, namely $\displaystyle u_t(\ell) = \frac{2t}{\ell(\ell+1)} + O(1)$.
Janson~\cite{janson}, extending a result of Mahmoud, Smythe and Szyma\'nski~\cite{MSS}, proved a central limit theorem for
$D_t(\ell)$ as $t \to \infty$, jointly for all $\ell \ge 1$.

A much more general model was introduced and studied by
Cooper and Frieze~\cite{cf03}, and Cooper~\cite{c05}: in the latter paper, Cooper proved a general result that implies
(weak) concentration for $D_t(\ell)$ whenever $\ell \le t^{1/6}/\log^2 t$.

\smallskip

The maximum degree $\Delta_t$ of the preferential attachment tree is known to behave as $t^{1/2}$ as $t \to \infty$.
M\'ori~\cite{mori} proved a law of large numbers and a central limit theorem for the $\Delta_t$: in particular, he showed that
$\Delta_t t^{-1/2}$ converges almost surely to some positive (non-constant) random variable, as $t \to \infty$.
Further, he showed that the fluctuations of $\Delta_t t^{-1/2}$ around the limit, scaled by $t^{-1/4}$, converge in
distribution to a normal law.

Many variants of the preferential attachment process have been studied.  The limiting proportion of vertices of each
{\em fixed} degree $\ell$ has been investigated in many different models extending and generalising that of preferential
attachment trees.  See for instance, Rudas, Toth and Valko~\cite{RTV}, Athreya, Ghosh and Sethuraman~\cite{AGS},
Deijfen, van den Esker, van der Hofstad and Hooghiemstra~\cite{dehh}, and Dereich and M\"orters~\cite{dm}.  See also
the survey of Bollob\'as and Riordan~\cite{br02} for a number of other results on related models.

\smallskip

Our principal aim in this paper is to prove concentration of measure results for $D_t(\ell)$ for all values of $\ell$ up to
the expected maximum degree.  For values of $\ell$ above $\eps t^{1/3}$, the expectation of $D_t(\ell)$ is of order at most~1,
and all we show is that $D_t(\ell)$ is, with high probability, at most about $\log t$.  For values of $\ell$ at most
$(t/\log t)^{1/3}$, we shall prove
%what should be essentially the optimal concentration result, namely
that $D_t(\ell)$ is concentrated within about $\sqrt{t\log t /\ell^3}$ of its mean.

We can write $D_t(\ell) = \sum_{s=1}^t I(s,t,\ell)$, where $I(s,t,\ell)$ is the indicator of the event that the vertex arriving
at time $s$ (or, for $s \le \tau_0$, the initial vertex labelled $s$) has degree exactly $\ell$ at time $t$.  One would expect that,
for large $t$ and $\ell$ in an appropriate range, most of the variables $I(s,t,\ell)$ are approximately independent of each other, each
with mean bounded away from~1.
This would suggest that the variance of $D_t(\ell)$ is of the same order as its mean, and that $D_t(\ell)$ should be concentrated within
about $\sqrt {\E D_t(\ell)} \simeq \sqrt{t/\ell^3}$ of its mean.  This is indeed the case for constant~$\ell$: see Janson~\cite{janson} for
asymptotic formulae for the covariances ${\rm Cov} (D_t(\ell)/t, D_t(j)/t)$.  So our
concentration result is likely to be best possible up to logarithmic factors.

Our methods can also be used to prove similar results for the random variable $U_t(\ell)$, the number of vertices of degree {\em at least}
$\ell$ at time $t$.  The expectation of $U_t(\ell)$ is approximately $2t/\ell^2$ for large $t$: at the end of this paper, we indicate
briefly how to adapt our proof to show that $U_t(\ell)$ is concentrated within about $\sqrt{t \log t/\ell^2}$ of its mean as long as
$\ell \le t^{1/2}/\log^{13/2} t$.

\smallskip

Before stating our results, we specify our model precisely.
We start at some time $\tau_0 \ge 1$, with an initial graph $G(\tau_0)=(V(\tau_0),E(\tau_0))$, with
$|V(\tau_0)| = \tau_0$, $|E(\tau_0)| = \tau_0 - 1$; we think of $G(\tau_0)$ as a tree, although it need not be.  At each step $t > \tau_0$,
a new vertex $v_t$ is created, and is joined to existing vertices by one new edge, whose other endpoint is chosen by
{\em preferential attachment}, that is, a vertex $v$ is chosen as an endpoint with probability proportional
to its degree at time $t-1$.  Note that, if $G(\tau_0)$ is a tree, then the graph at all later stages is also a tree.

Our main theorem concerns the number $D_t(\ell)$ of vertices of degree exactly $\ell$ at time $t$, for all $\ell \ge 1$ and
$t \ge \tau_0$.

\begin{theorem} \label{thm:main}
Let $\tau_0 \ge 4$ and $\psi \ge 10^5 \sqrt{\tau_0-1} \log^3\tau_0$ be constants.
Let $G(\tau_0)$ be any graph with $\tau_0$ vertices and $\tau_0 -1$ edges, and consider the
preferential attachment process with initial graph $G(\tau_0)$ at time $\tau_0$, and the associated Markov chain
$D = (D_t (\ell): t \ge \tau_0, \ell \in \N)$.

With probability at least $1 - \frac{4}{\psi}$, we have
$$
\left| D_t(\ell) - \frac{4t}{\ell(\ell+1)(\ell+2)} \right| \le 120 \sqrt{\frac{t \log (\psi t)}{\ell^3}} + 301 \psi^2 \log(\psi t),
$$
for all $\ell \ge 1$, and all $t \ge \tau_0$.
\end{theorem}

The parameter $\psi$ is a constant which may be chosen arbitrarily large in order to make the
probability of failure arbitrarily small; the results are only of interest when $t$ is larger than some $t_0(\psi)$.
All our results are stated in terms of such a parameter (denoted $\psi$ or $\omega$).

We state here an analogous result about the number $U_t(\ell)$ of vertices of degree {\em at least} $\ell$ at time $t$,
for all $\ell \ge 2$ and all $t \ge \tau_0$.  (Note that $U_t(1)$ is equal to $t$ for each $t \ge \tau_0$.)

\begin{theorem} \label{thm:main-U}
Let $\tau_0 \ge 4$ and $\psi \ge \max \left( \tau_0, 10^5 \sqrt{\tau_0-1} \log^3\tau_0\right)$ be constants.
Let $G(\tau_0)$ be any graph with $\tau_0$ vertices and $\tau_0 -1$ edges, and consider the
preferential attachment process with initial graph $G(\tau_0)$ at time $\tau_0$, and the associated Markov chain
$U = (U_t (\ell): t \ge \tau_0, \ell \in \N)$.

With probability at least $1 - \frac{4}{\psi}$, we have
$$
\left| U_t(\ell) - \frac{2t}{\ell(\ell+1)} \right| \le 45 \frac{\sqrt{t \log (\psi t)}}{\ell} + 4 \times 10^9 \psi \log^7(\psi t),
$$
for all $\ell \ge 2$, and all $t \ge \tau_0$.
\end{theorem}

We do not give a detailed proof of Theorem~\ref{thm:main-U} in this paper, but we do give an indication of how to adapt
our proof of Theorem~\ref{thm:main} to give this result.
The term $\log^7(\psi t)$ appearing above could certainly be improved with more work.
%in our proof of Theorem~\ref{thm:main-U}, we have
%chosen to re-use lemmas from the proof of Theorem~\ref{thm:main} where possible, at the expense of this higher than necessary power
%of the logarithm.

Theorem~\ref{thm:main-U} seems to be the first explicit result concerning concentration of measure for $U_t(\ell)$, although some weak
concentration can be deduced from concentration results for $D_t(\ell)$.  Moreover, Talagrand's inequality~\cite{talagrand} can be applied
readily: to demonstrate that $U_t(\ell) \ge x$, a certificate of length at most $O(x\ell)$ suffices (see for instance~\cite{cmcd98} for
details of the method).  This method gives concentration for $U_t(\ell)$ up to about $\ell = t^{1/3}$, and indeed concentration for
$D_t(\ell)$ up to about $\ell = t^{1/5}$.

For constant values of $\ell$, Bollob\'as, Riordan, Spencer and Tusn\'ady~\cite{brst01} showed that $D_t(\ell)$ is concentrated within
about $t^{1/2}$ of its mean, which is best possible; a similar result for $U_t(\ell)$ follows.  For larger values of $\ell$,
in particular where $\ell$ is growing as a small power of $t$, earlier methods (including the method based on Talagrand's
inequality that we mentioned above) do not give the ``optimal'' concentration of $D_t(\ell)$ or $U_t(\ell)$
about their respective means.  Our results above do give what should be optimal concentration, up to
logarithmic factors, for $D_t(\ell)$ and $U_t(\ell)$, whenever the expectations of these random
variables tend to infinity, again up to logarithmic factors.

\smallskip

In Section~\ref{Sexpomart}, we give an exposition of a method based on exponential supermartingales, that is widely used in
the analysis of continuous time Markov processes.  We transfer the method to the discrete time setting, and state two theorems
that we shall use, and that can be applied in other similar contexts.

In Section~\ref{Seve}, we apply our method to describe the evolution of the degree of a fixed vertex in the preferential
attachment model.  We do this partly to illustrate the method, but mostly so that we can use the results in later sections.
We prove a result on the maximum degree $\Delta_t$ that is weaker than M\'ori's~\cite{mori}, but simple to prove, in the interests
of keeping the paper self-contained.

Sections~\ref{Ssimple} to~\ref{Slemma} are devoted to the proof of Theorem~\ref{thm:main}.  Section~\ref{Ssimple} contains the
main thread of the proof, and we defer some calculations to Sections~\ref{Sbounds} and~\ref{Slemma}.  One difficulty we face is that
we cannot get sharp results by working directly with the natural martingale associated to the Markov chain
$D=(D_t(\ell) : t \ge \tau_0, \ell \in \N)$, so we work instead with a suitable transform of that
martingale.  Proving concentration of measure for the transform is not straightforward, so we introduce another Markov process
derived from $D$, and apply our methods from Section~\ref{Sexpomart} to that process.

Section~\ref{Usimple} contains a brief sketch of the proof of Theorem~\ref{thm:main-U}.

%Sections~\ref{Usimple} and~\ref{Ubounds} are devoted to the proof of Theorem~\ref{thm:main-U}.  The proof is extremely similar to that
%of Theorem~\ref{thm:main}, and we omit some of the details.

\smallskip

In this paper, we deal only with the preferential attachment tree.  However, our methods will extend to more general settings, and
indeed we believe we can prove results similar to those above for the general Cooper-Frieze model.  We intend to address this elsewhere;
in a very brief final section, we make a few remarks on the difficulties involved in extending our proof to other preferential attachment models.

\section{Our method: exponential supermartingales}

\label{Sexpomart}

The following technique is adapted from a fairly standard method used in the analysis of continuous-time random processes; see
for instance \cite{dn}, \cite{l03} and~\cite{ln05}.  We have not been able to find a suitable account in the literature of a
discrete-time, and time-dependent, version of the method for us to quote, so we develop the theory here.  We provide
results that we hope may prove useful in other settings.

Let $X = (X_t: t \in \bbZ^+)$ be a discrete-time Markov chain, possibly time non-homogeneous,
with countable state space $E$ and transition matrix $P_t= (P_t(x,x'): x,x' \in E)$ at time $t$.
(Here and in what follows, our matrices -- which will normally be infinite -- have rows and
columns indexed by the countable set $E$.)
Let $(\cF_t)$ be a filtration, and suppose that $(X_t)$ is adapted.

Let $I$ denote the identity matrix.  Further, let us write, for a matrix $A$,
\begin{eqnarray*}
(A f) (x) = \sum_{x' \in E} A (x,x') f(x').
\end{eqnarray*}
Then we see that
\begin{eqnarray*}
[(P_t-I) f](x) = \sum_{x' \in E} [P_t(x,x') - I(x,x')] f(x') = \sum_{x' \in E}
P_t(x,x') (f(x') -f(x)).
\end{eqnarray*}
Further, note that
\begin{eqnarray}
\label{eq.1}
[(P_t-I)f](x) = \E [f(X_{t+1}) -f(x) \mid X_t =x],
\end{eqnarray}
that is, $[(P_t-I)f](x)$ is the expected change in $f$ at the $t$-th step given that $X_t = x$.

\begin{lemma}
\label{lem.mart}
Suppose $X_0=x_0$ a.s.
Let $f: E \to \bbR$ be a function such that
$\E [|f(X_t)| \mid X_0=x_0]$ is finite for each $t$. Then
\begin{eqnarray*}
M^f_t & = & f(X_t) -f(X_0) - \sum_{s=0}^{t-1}[(P_s -I) f](X_s) \\
& = & f(X_t) - f(X_0) - \sum_{s=0}^{t-1} \sum_{x'} P_s(X_s,x')(f(x') - f(X_s))
\end{eqnarray*}
is an $(\cF_t)$-martingale.
\end{lemma}
\begin{proof}
The proof for the time homogeneous case can be found in Norris~\cite{n}.
Checking that $M^f_t$ is a martingale in the time non-homogeneous case is just as easy.
Consider
\begin{eqnarray*}
\E [ M^f_{t+1} \mid \cF_t] & = & \E \left[ f(X_{t+1}) - f(X_0) - \sum_{s=0}^t
[(P_s -I )f] (X_s) \mid \cF_t \right]\\
& = & \E [ f(X_{t+1}) \mid X_t ] - f(X_0) - [(P_t -I) f] (X_t) \\
&&\mbox{} - \sum_{s=0}^{t-1} [ (P_s -I) f] (X_s) \\
& = & f(X_t)  - f(X_0) -
\sum_{s=0}^{t-1} [ (P_s -I) f] (X_s) \\
& = & M^f_t,
\end{eqnarray*}
where we used~(\ref{eq.1}).  Also, for each $t \ge 0$,
\begin{eqnarray*}
\E |M^f_t| & \le & \E |f(X_t)| + \E |f(X_0)| + \sum_{s=0}^{t-1} \E |[(P_s-I)] f(X_s)| \\
& \le & \E |f(X_t)| + \E |f(X_0)| + \sum_{s=0}^{t-1} \left( \E |f(X_s)| + \E |f(X_{s+1})| \right) \\
& < & \infty .
\end{eqnarray*}
\end{proof}

\begin{lemma}
\label{lem.supmart}
Suppose $X_0=x_0$ a.s.
Let $f: E \to \bbR^+$ be a function.
Then
\begin{eqnarray*}
Z^f_t = \frac{f(X_t)}{f(X_0)} \exp \Big ( - \sum_{s=0}^{t-1}\frac{[(P_s -I)f](X_s)}{f(X_s)} \Big )
\end{eqnarray*}
is an $\cF_t$-supermartingale, as long as $\E Z^f_t < \infty$ for all $t$.
\end{lemma}
\begin{proof}
Consider
\begin{eqnarray*}
\E [Z^f_{t+1} \mid \cF_t] & = &\E [ f(X_{t+1}) \mid X_t]
\frac{1}{f(X_0)} \exp \Big ( - \sum_{s=0}^{t}\frac{[(P_s -I) f](X_s)}{f(X_s)} \Big )\\
& = & \frac{f(X_t)}{f(X_0)}\frac{(P_tf)(X_t)}{f(X_t)}
\exp \Big (1- \frac{(P_tf)(X_t)}{f(X_t)} \Big )\\
&&\mbox{}  \times \exp \Big ( - \sum_{s=0}^{t-1}\frac{[(P_s -I) f](X_s)}{f(X_s)} \Big )\\
& \le &  \frac{f(X_t)}{f(X_0)} \exp \Big (-1+ \frac{(P_tf)(X_t)}{f(X_t)}\Big )
\exp \Big (1- \frac{(P_tf)(X_t)}{f(X_t)} \Big ) \\
&&\mbox{} \times \exp \Big ( - \sum_{s=0}^{t-1}\frac{[(P_s -I) f](X_s)}{f(X_s)} \Big )\\
& = & Z^f_t,
\end{eqnarray*}
where we have used the fact that $\E [ f(X_{t+1}) \mid X_t]=P_tf$, and the fact that $x \le \exp (-1+x)$ for all $x$.
\end{proof}

Note that, for a continuous-time Markov chain, the analogue of $Z^f_t$ in  Lemma~\ref{lem.supmart} is in fact a
martingale; see for example Lemma~3.2 in Chapter~4 of~\cite{ek}. In the time-continuous case, the matrix
$(P_t-I)$ is replaced by the generator matrix $A_t$ of the Markov chain, which is
the derivative at time $t$ of its transition semigroup $P_t$.

\smallskip

We shall show how, under certain conditions, Lemma~\ref{lem.mart} and Lemma~\ref{lem.supmart} can
be used to prove a law of large numbers for a Markov chain.
\begin{lemma}
\label{lem.exp-mart-1}
Let $g: E \to \R$ be a function, and suppose that $X_0=x_0$ a.s., for some $x_0 \in E$.
For $\theta \in \R$, let
$$
\phi^g_s (x, \theta) = \sum_{x' \in E} P_s (x,x') \Big (e^{\theta (g(x') - g(x))} - 1- \theta (g(x') - g(x)) \Big ).$$
Then
$$
Z^g_t (\theta) = \exp \Big (\theta M^g_t - \sum_{s=0}^{t-1} \phi^g_s (X_s,\theta ) \Big )
$$
is an $\cF_t$-supermartingale, as long as $\E Z^g_t (\theta)  < \infty$ for each $t$.
\end{lemma}

\begin{proof}
The result is a consequence of
Lemma~\ref{lem.supmart}, with $f(x) = e^{\theta (g(x)-g(x_0))}$.  That lemma tells us that
$Z^f_t$ is a supermartingale, and we need only verify that $Z^f_t = Z^g_t(\theta)$ for this choice
of $f$.

The calculation goes as follows:
\begin{eqnarray*}
Z^f_t & = & \frac{f(X_t)}{f(X_0)} \exp \Big ( - \sum_{s=0}^{t-1}\frac{[(P_s -I) f](X_s)}{f(X_s)}\Big )\\
& = & \exp ( \theta (g(X_t) - g(X_0) )) \\
&& \mbox{}\times \exp \Big ( - \sum_{s=0}^{t-1}
\frac{\sum_{x'} P_s (X_s,x') [ e^{\theta (g(x')-g(X_0))} - e^{\theta (g(X_s) - g(X_0))}]}{e^{\theta (g(X_s) -g(X_0))}} \Big )\\
& = & \exp \Big ( \theta (g(X_t) -g(X_0)) - \sum_{s=0}^{t-1} \sum_{x'} P_s (X_s,x')
[e^{\theta (g(x') -g(X_s))} -1 ] \Big )\\
& = &  \exp \Big ( \theta (g(X_t) -g(X_0)) - \theta \sum_{s=0}^{t-1} \sum_{x'} P_s (X_s,x')
(g(x') -g(X_s)) \\
&& \mbox{} - \sum_{s=0}^{t-1} \phi^g_s (X_s,\theta)   \Big )\\
& = & \exp \Big (\theta M_t^g - \sum_{s=0}^{t-1} \phi^g_s (X_s,\theta) \Big ).
\end{eqnarray*}
\end{proof}

Note that, while $X_t$ remains in a `good' set $S_t$ of states $x$ where $e^{\theta (g(x')-g(x))}$ is bounded by some constant (possibly
depending on $t$) over all $x'$ such that $P_t (x,x') > 0$ and all $x \in S_t$, (i.e., the size of changes in $g$ stays uniformly bounded),
then the finiteness assumption of Lemma~\ref{lem.exp-mart-1} holds.  Furthermore, we can approximate $e^{\theta (g(x')-g(X_t))}$ using a
Taylor expansion.

In many applications, in particular those in this paper, $|g(x')-g(x)|$ will be
uniformly bounded over the entire state space $E$ and over all transition matrices $P_t$: if we work up to some fixed
time $\tau$, then it suffices to have the bound valid for $t< \tau$.  We assume from now on that, for every $\tau \ge 0$, there is
some real number $J = J(\tau)$ such that $g$ satisfies:
\begin{equation} \label{eq:bounded}
\sup_{s < \tau ,x} \sup_{x': P_s (x,x') \not = 0} |g(x') - g(x)| \le J < \infty.
\end{equation}

Now we fix some real number $\alpha > 0$, and restrict attention to values of $\theta$ such that $|\theta | \le \alpha$.
We use the identity
$$
e^z - 1 - z = z^2 \int_{r=0}^1 e^{rz} (1-r) \, dr
$$
to deduce that
\begin{eqnarray*}
\phi^g_s (x,\theta) & = & \sum_{x'} P_s (x,x') \theta^2 (g(x')-g(x))^2 \int_0^1 e^{r \theta (g(x')-g(x))} (1-r)\, dr\\
&\le & \theta^2 \sum_{x'} P_s(x,x') (g(x')-g(x))^2 e^{\alpha J} \int_0^1 (1-r) \, dr \\
& = & \frac12 \theta^2 e^{\alpha J} \sum_{x'} P_s(x,x') (g(x') - g(x))^2 .
\end{eqnarray*}

\smallskip

Suppose that $X_0=x_0$ a.s., for some $x_0 \in E$, and that we study the chain up to some
time $\tau > 0$.  Our aim is to show that $M_t^g$ remains small over the period $0\le t \le \tau$.
For a precise statement, we need a few more definitions.

We set
$$
\Phi^g_t(X) = \sum_{s=0}^t\sum_{x'} P_s(X_s,x') (g(x')-g(X_s))^2,
$$
so that
$$
Z^g_t(\theta) \ge \exp \left( \theta M^g_t - \frac12 \theta^2 e^{\alpha J} \Phi^g_{t-1}(X) \right),
$$
for all $\theta$ with $|\theta| \le \alpha$.

Now let $R$ be a positive real number, and set
\begin{eqnarray*}
T_R = \inf \{t \ge 0: \Phi^g_t(X) > R \}.
\end{eqnarray*}
Thus, for $t \le T_R$, we have $\Phi^g_{t-1}(X) \le R$, and therefore
$$
Z^g_t(\theta) \ge \exp \Big (\theta M^g_t - \frac12 \theta^2 e^{\alpha J} R \Big ),
$$
provided $|\theta| \le \alpha$.

Also, for $\delta>0$, we define
\begin{eqnarray*}
T_g^+(\delta) = \inf \{t: M^g_t > \delta \},
\quad T_g^-(\delta) = \inf \{t: M^g_t < -\delta \},
\end{eqnarray*}
and
\begin{eqnarray*}
T_g(\delta) = T_g^+(\delta) \wedge T_g^-(\delta) = \inf \{ t : |M^g_t| > \delta \}.
\end{eqnarray*}

\smallskip

\begin{lemma}
\label{lem.key-1}
Fix $\tau > 0$ and $R >0$, and let $g: E \to \R$ be a function satisfying {\rm (\ref{eq:bounded})} for some $J \in \R$.
Also, let $\alpha > 0$ and $\delta >0$ be any constants such that $\delta \le e^{\alpha J} \alpha R$.  Then
\begin{eqnarray*}
\P \Big (T_g (\delta) \le T_R \land \tau \Big) \le  2e^{-\delta^2/(2 R e^{\alpha J}) },
\end{eqnarray*}
and hence
\begin{eqnarray*}
\P \Big ((\sup_{0 \le t \le \tau} |M_t^g| > \delta) \land (T_R \ge \tau) \Big)
\le  2e^{-\delta^2/(2 R e^{\alpha J}) }.
\end{eqnarray*}
In particular:
\begin{itemize}
\item[(i)] for any $\omega \le R/J^2$, we obtain the following by choosing
$\alpha = \log 2/J$ and $\delta = \sqrt{\omega R}$:
\begin{eqnarray*}
\lefteqn{\P \left(\left(\sup_{0 \le t \le \tau} |M_t^g| > \sqrt{\omega R}\right) \land (T_R \ge \tau) \right)} \\
&\le& \P \Big( T_g \big(\sqrt{\omega R}\big) \le T_R \land \tau \Big) \\
&\le& 2e^{-\omega/4};
\end{eqnarray*}
\item[(ii)] for any $\omega \ge R/J^2$, we set
$\alpha =\frac{1}{J} \log \left( 2\omega J^2/R \right)$ and $\delta = \omega J$, and obtain
\begin{eqnarray*}
\P \Big ((\sup_{0 \le t \le \tau} |M_t^g| > \omega J) \land (T_R \ge \tau) \Big)
\le \P \Big (T_g (\omega J) \le T_R \land \tau \Big)
\le 2e^{-\omega/4}.
\end{eqnarray*}
\end{itemize}
\end{lemma}

\begin{proof}
Fix any real $\theta$ with $|\theta| \le \alpha$.  For ease of notation, we write $T_g^+$ for $T_g^+(\delta)$ and
$T_g^-$ for $T_g^-(\delta)$.  By Lemma~\ref{lem.exp-mart-1}, $(Z_t^g(\theta))_{t \ge 0}$ is a supermartingale.

On the event $\{T_g^+ \le T_R \land \tau\}$, we have $M^g_{T_g^+} > \delta$ and
$$
Z^g_{T^+_g}(\theta) > \exp\left( \theta \delta - \frac12 \theta^2 e^{\alpha J} R \right).
$$
By optional stopping,
\begin{eqnarray*}
\E [ Z^g_{T^+_g}(\theta)] \le \E [Z^g_0(\theta)] = 1.
\end{eqnarray*}
Hence, using the Markov inequality,
\begin{eqnarray*}
\P ( T_g^+ \le T_R \land \tau) & \le & \P \left(Z^g_{T^+_g}(\theta) >
\exp \left(\delta \theta - \frac12 \theta^2 e^{\alpha J} R  \right) \right) \\
& \le  & \exp\left( -\delta \theta + \frac12 \theta^2 e^{\alpha J} R \right).
\end{eqnarray*}
Optimising in $\theta$, we find that $\theta = \delta/ e^{\alpha J} R$ is the best choice,
and note that $|\theta | \le \alpha$. This yields
\begin{eqnarray*}
\P ( T_g^+ \le T_R \land \tau) \le  \exp \left(-\delta^2 \Big/2 e^{\alpha J} R \right).
\end{eqnarray*}
An almost identical calculation gives
\begin{eqnarray*}
\P ( T_g^- \le T_R \land \tau) \le  \exp \left(-\delta^2 \Big/2 e^{\alpha J} R \right),
\end{eqnarray*}
and the first part of the result follows.

The two special cases are obtained by choosing the given values of $\alpha$ and $\delta$, and verifying
that
\begin{eqnarray*}
\delta \le e^{\alpha J} \alpha R \quad \mbox{ and } \quad
\frac{\delta^2}{2 e^{\alpha J} R} \ge \frac {\omega}{4}
\end{eqnarray*}
in each case.
\end{proof}

We summarise what we have proved in a theorem.

\begin{theorem} \label{thm.key}
Let $X = (X_t)_{t\in \Z^+}$ be a discrete-time Markov chain, with countable state space $E$ and transition
matrix $P_t$ at time $t$, and suppose that $(X_t)$ is adapted to a filtration $(\cF_t)$.
Let $g: E \to \R$ be any function, $\tau$ any natural number, and $J$ any real number, satisfying
$$
\sup_{s < \tau, x \in E} \sup_{x'\in E: P_s(x,x') > 0} |g(x')-g(x)| \le J.
$$
Set
$$
\Phi^g_t (X) = \sum_{s=0}^t \sum_{x'\in E} P_s(X_s,x') \big(g(x')-g(X_s)\big)^2.
$$
Let $R >0$ be a real number, and set
$$
T_R = \inf \{ t \ge 0 : \Phi^g_t(X) > R \}.
$$
Then
$$
M^g_t = g(X_t) - g(X_0) - \sum_{s=0}^{t-1}\sum_{x'\in E} P_s(X_s,x')(g(x') - g(X_s))
$$
is an $(\cF_t)$-martingale and, for any $\omega > 0$,
$$
\P \left( \left( \sup_{0\le t\le \tau} |M^g_t| > \max\left(\sqrt{\omega R}, \omega J\right) \right)
\wedge (T_R \ge \tau) \right) \le 2 e^{-\omega/4}.
$$
\end{theorem}

We have demanded that the state space be countable, so that we can express our results in terms of sums over the
state space.  It suffices to assume instead that, for any state $x\in E$, and any time $s$, there is a countable set
$E(x,s) \subseteq E$ such that $\sum_{x' \in E} P_s(x,x') = 1$.  Indeed, under this assumption, if $X_0 = x_0$ a.s.\ for some state
$x_0$, then there is a countable set $E' \subseteq E$ such that, a.s., $X_t \in E'$ for all $t\ge 0$.

We also remark that, in the statement above, we begin our consideration of the chain at time~0.  When applying Theorem~\ref{thm.key}
in the analysis of the preferential attachment tree, we
shall instead start at some fixed time $\tau_0$: of course this makes no substantive difference.

\smallskip

In some instances, for example in Section~\ref{Seve}, we will want to bound the probability that $|M_t^g| \le \delta(t)$ for all
$t \le \tau$, where $\delta(t)$ is a suitable function growing with $t$.  One easy way to do this is to apply the above
theorem for each value $t \le \tau$, choosing an appropriate value $R(t)$ of $R$ at each time.  This approach has the drawback
that it is necessary to sum the probabilities of failure over $t \le \tau$.  Better bounds may be obtained by applying the lemma only
for a sparse sequence of values $t$, as we illustrate in the proof of the following result.

\smallskip

The notation here is essentially as for Theorem~\ref{thm.key}.  We again have a real-valued function $g$ defined on
the state space $E$ of a Markov chain $X$, and the change in $g$ is uniformly bounded by $J$ over all possible transitions of the chain.
The function $\Phi^g_t(X)$ is as in Theorem~\ref{thm.key}.  Now we have a non-decreasing function $R: \Z^+ \to \R^+$, and we set
$$
T_R = \inf\{ t\ge 0 : \Phi^g_t(X) > R(t) \}.
$$
Also, for any non-decreasing function $\delta: \Z^+ \to \R^+$, we define an associated stopping time
$$
T_g(\delta) = \inf\{ t\ge 0: |M_t^g| > \delta(t) \}.
$$
With the notation as above, we have the following result.

\begin{theorem} \label{thm.sparse}
\begin{itemize}
\item [(a)]
Fix $\omega > 4$, and let $\delta(t) = \max(\omega J, 2 \sqrt{\omega R(t-1)})$ for $t \ge 1$.
Then, for any $\tau>0$ such that $R(\tau-1) \ge \omega J^2$,
$$
\P (T_g(\delta) \le T_R \wedge \tau ) \le 2 \log \left(\frac{8 R(\tau-1)}{\omega J^2}\right) e^{-\omega/4}.
$$
\item [(b)] Fix $\psi \ge 4/J^2$, and suppose that $R(t)$ tends to infinity as $t\to \infty$.
For $t \ge 1$, let $\widetilde\delta(t) = 2 \max\left(\psi J \log (\psi J^2), \sqrt{\psi R(t-1)\log R(t-1)}\right)$.  Then
$$
\P (T_g(\widetilde\delta) \le T_R) \le 5 e^{-\psi/4}.
$$
\end{itemize}
\end{theorem}

\begin{proof}
For~(a), we define a finite sequence of times $\tau_1, \tau_2, \dots$ as follows.
Let $\tau_1$ be the first $t$ for which $R(t) > \omega J^2$: by assumption $\tau_1 \le \tau$.
For $j > 1$, if $\tau_j < \tau$ then we set
$$
\tau_{j+1} = \inf \{ t > \tau_j : R(t) > 4 R(\tau_j)\} \land \tau.
$$
The final term $\tau_N$ in the sequence is the first $\tau_j$ with $\tau_j = \tau$, and the number $N$ of
terms in the sequence is then no greater than
$2 + \log_4 \left(\frac{R(\tau-1)}{\omega J^2}\right) \le \log \left(\frac{8 R(\tau-1)}{\omega J^2}\right)$.

We first apply Lemma~\ref{lem.key-1}(ii) with $\tau = \tau_1$ and $R=R(\tau_1-1)$, noting that
$\omega \ge R(\tau_1-1)/J^2$ by definition of $\tau_1$.  We obtain that
$$
\P (T_g(\omega J) \le T_R \land \tau_1) \le \P (T_g(\omega J) \le T_{R(\tau_1-1)} \land \tau_1) \le 2e^{-\omega/4}.
$$
As $\delta(t) \ge \omega J$ for all $t \le \tau_1$, this means that, with probability at least
$1-2e^{-\omega/4}$, $|M_t^g| \le \delta(t)$ for all times $t \le T_R \land \tau_1$.

\smallskip

For each of the times $\tau = \tau_j$ ($j \ge 2$), we apply Lemma~\ref{lem.key-1}(i) with $R=R(\tau_j-1)$,
noting now that $\omega \le R(\tau_j-1)/J^2$ by choice of $\tau_1$.  We obtain that
$$
\P \left( T_g\left(\sqrt{\omega R(\tau_j-1)}\right) \le T_R \land \tau_j \right) \le 2 e^{-\omega/4}.
$$
For each $t$ with $\tau_{j-1} < t \le \tau_j$, we have
$$
\delta(t) \ge \delta (\tau_{j-1}+1) \ge 2 \sqrt{\omega R(\tau_{j-1})} \ge \sqrt{\omega R(\tau_j -1)},
$$
since $R(\tau_j-1) \le 4 R(\tau_{j-1})$ by definition of $\tau_j$.
We conclude that, for each $j\ge 2$, with probability at least $1-2e^{-\omega/4}$,
$|M_t^g| \le \delta(t)$ for all times $t$ with $\tau_{j-1} < t \le T_R \land \tau_j$.

It now follows that, with probability at least $1-2Ne^{-\omega/4}$,
we have $|M_t^g| \le \delta(t)$ for all times $t \le T_R \land \tau$, and part~(a) follows.

\smallskip

The proof of part~(b) is very similar in style.  This time we let $\tau_1$ be the first $t$ for which
$R(t) > 2 \psi J^2 \log (\psi J^2)$.  Given $\tau_j$, we let $\tau_{j+1}$ be the minimum $t$
such that $R(t) > 4 R(\tau_j)$.  The assumption that $R(t)$ tends to infinity ensures that we obtain an
infinite sequence $(\tau_j)_{j\ge 1}$ of times.

We apply Lemma~\ref{lem.key-1}(ii) with $\omega = 2 \psi \log (\psi J^2)$, $\tau = \tau_1$, and
$R=R(\tau_1-1)$, noting that $\omega \ge R(\tau_1-1)/J^2$ by choice of $\tau_1$.  We obtain that
$$
\P \left( T_g(2\psi J \log (\psi J^2)) \le T_R \land \tau_1 \right) \le 2e^{-\psi/4}.
$$
Since $\widetilde\delta(t) \ge 2 \psi J \log (\psi J^2)$ for all $t$, this implies that, with probability
at least $1-2e^{-\psi/4}$, $|M_t^g| \le \widetilde\delta(t)$ for all times $t \le T_R \land \tau_1$.

\smallskip

Now, for each $\tau=\tau_j$ ($j\ge 2$), we apply Lemma~\ref{lem.key-1}(i) with $R = R(\tau_j-1)$ and
$\omega = \psi \log R(\tau_{j-1})$ for each $j$.  We need to check that $\omega \le R(\tau_j-1)/J^2$, i.e.,
that $R(\tau_j-1)/ \log R(\tau_{j-1}) \ge \psi J^2$:
we have
$$
\frac{R(\tau_j-1)}{\log R(\tau_{j-1})} \ge \frac{R(\tau_{j-1})}{\log R(\tau_{j-1})} \ge
\frac{2\psi J^2 \log(\psi J^2)}{\log(2 \psi J^2 \log(\psi J^2))} > \frac{2 \psi J^2 \log(\psi J^2)}{2\log(\psi J^2)}
= \psi J^2,
$$
as required; here we used the facts that (i)~$R(\tau_j -1) \ge R(\tau_{j-1}) \ge 2 \psi J^2 \log(\psi J^2)$, by the definition of
the $\tau_j$, and $R(\cdot)$ is increasing, (ii)~$\psi J^2 \ge 4$ and $x/\log x$ is an increasing function,
with minimum value~$e$, for $x \ge e$ (so in particular $2\log(\psi J^2) < \psi J^2$).

We obtain:
$$
\P \left( T_g\left(\sqrt{\psi R(\tau_j-1) \log R(\tau_{j-1})}\right) \le T_R \land \tau_j \right)
\le 2 e^{-\psi \log R(\tau_{j-1})/4}.
$$
We have, for $\tau_{j-1} < t \le \tau_j$,
$$
\widetilde\delta(t) \ge \widetilde\delta(\tau_{j-1}) \ge 2 \sqrt{\psi R(\tau_{j-1}) \log R(\tau_{j-1})}
\ge \sqrt{\psi R(\tau_j -1) \log R(\tau_{j-1})}
$$
and so, with probability at least $1 -  2 e^{-\psi \log R(\tau_{j-1})/4}$,
$|M_t^g| \le \widetilde\delta(t)$ for all times $t$ with $\tau_{j-1} < t \le T_R \land \tau_j$.

Therefore, summing over $j$, and noting that $R(\tau_{j-1}) \ge 2 \psi J^2 \log (\psi J^2) 4^{j-2} > e^{j-1}$
for each $j \ge 2$,
\begin{eqnarray*}
\P \left( |M_t^g| > \widetilde\delta(t) \mbox{ for some } t \le T_R \right)
&\le& 2 e^{-\psi/4} + 2 \sum_{j=2}^\infty e^{-\psi(j-1)/4} \\
&=& 2 e^{-\psi/4} \left( 1 + \frac{1}{1-e^{-\psi/4}} \right)
\le 5 e^{-\psi/4},
\end{eqnarray*}
as desired.
\end{proof}

\section{Evolution of the degree of a vertex}   \label{Seve}

For the remainder of the paper, we will use the results of the previous section to analyse various aspects of
the preferential attachment process.

Our first, relatively simple, application is to the evolving degree of a single vertex; loosely, we prove that, if a vertex
has degree $k$ at time~$s$, its degree at a later time $t$ is unlikely to be far from $k \sqrt{t/s}$.  Results of a similar
flavour are in the literature already (see for instance~Cooper~\cite{c05}, Athreya, Ghosh and Sethuraman~\cite{AGS} and
Dereich and M\"orters~\cite{dm}); we
give them here partly to illustrate our methods and partly because we shall have need of the results from this section later on.

We assume as always that our process starts at some time $\tau_0$, with $\tau_0$ vertices and $\tau_0-1$ edges.  At each stage,
one new vertex and one new incident edge are created, so that at each time $s \ge \tau_0$ there are $s$ vertices and $s-1$ edges.
We identify the vertex set at time $\tau_0$ with the set $[\tau_0] = \{ 1, \dots, \tau_0\}$, and label the new vertex arriving
at each later time $s$ with $s$, so that the set of vertices present at time $t \ge \tau_0$ is exactly $[t]$.

For a vertex $v$, and $t \in \bbN$ with $t \ge \max(\tau_0,v)$, let $X_t(v)$ be the degree of vertex $v$ at time $t$.
For $\tau_0 \le t < v$, we set $X_t(v)=0$.

Set $(X_t) = (X_t(v): v=1,2, \ldots )$; it is easily seen that $(X_t)$ is a Markov process.  Indeed, each component
$X_t(v)$ is separately a Markov process.  For $v \in \bbN$, we take $f_v(x) = x(v)$,
the degree of vertex $v$ in state $x$; the function $f_v$ is the projection of the state onto its $v$-th component.
In what follows, we shall assume that $v \le \tau_0$, so that vertex $v$ is present in the graph at time $\tau_0$, and we let
$m_0 = X_{\tau_0}(v)$ be its degree at the initial time.

We now want to calculate the corresponding martingale from Lemma~\ref{lem.mart}. First we note that
\begin{eqnarray*}
\E [f_v (X_{t+1}) - f_v (X_t) \mid X_t = x]  &=&  [(P_t-I) f_v ] (x)\\
&=& \frac{x(v)}{2(t-1)}.
\end{eqnarray*}
This is because the sum of all vertex degrees at time $t$ is $2(t-1)$, and the probability that a vertex $w$ is chosen as
the endpoint of the new edge from vertex $t$ at time $t$, conditional on $X_t=x$, is proportional to its degree $x(w)$, and therefore
the conditional probability that vertex $v$ is chosen is $x(v)/2(t-1)$.

By Lemma~\ref{lem.mart}, we know that the process $M(v)$ given by
\begin{eqnarray*}
M_t(v) &=& f_v(X_t) -f_v(X_{\tau_0}) - \sum_{s=\tau_0}^{t-1}[(P_s -I) f_v](X_s) \\
&=& X_t(v) - m_0 - \sum_{s=\tau_0}^{t-1} \frac{X_s(v)}{2(s-1)}
\end{eqnarray*}
is a martingale.  We re-write the above as
\begin{equation} \label{Xtv}
X_t(v) = M_t (v) + m_0 + \sum_{s=\tau_0}^{t-1} \frac{X_s(v)}{2(s-1)}.
\end{equation}

\smallskip

Let $x_t$ solve the recurrence relation
\begin{eqnarray*}
x_{t+1} = x_t \left( 1 + \frac{1}{2(t-1)} \right),
\end{eqnarray*}
for $t \ge \tau_0$, with $x_{\tau_0}=m_0$.
A simple induction argument shows that, for all $t \ge \tau_0$,
$$
m_0 \sqrt{ \frac{t-1}{\tau_0-1}} \le x_t \le m_0 \sqrt{ \frac{t-2}{\tau_0-2}}.
$$
Provided $\tau_0 \ge 4$, we have
$$
\sqrt{ \frac{t-2}{\tau_0-2}}
\le \sqrt{ \frac{t-1}{\tau_0-1}}
\left(1+ \frac{1}{2(\tau_0-2)}\right) \le \frac54 \sqrt{\frac{t-1}{\tau_0-1}},
$$
and so
\begin{equation} \label{xt}
x_t \le \frac54 m_0 \sqrt{\frac{t-1}{\tau_0-1}}.
\end{equation}

\smallskip

Now fix any $\omega \ge 4$.
For a vertex $v$, we define the time
$$
T_v = \inf \left\{t \ge \tau_0: X_t(v) > 60 \omega^3 m_0 \sqrt{\frac{t-1}{\tau_0-1}} \right\}.
$$
Then for $\tau_0 \le t < T_v$, we have $X_t(v) \le 60 \omega^3 m_0\sqrt{\frac{t-1}{\tau_0 -1}}$.

Let $E_t = X_t (v) -x_t$, with $E_{\tau_0}=0$; we want to bound $|E_t|$.
Substituting $X_t(v) = x_t + E_t$ in~(\ref{Xtv}), and using the recurrence
\begin{eqnarray*}
x_t  =  x_{\tau_0} + \sum_{s=\tau_0}^{t-1} \frac{x_s}{2(s-1)},
\end{eqnarray*}
we obtain that, for $\tau_0 \le t$,
\begin{eqnarray*}
|E_t| \le |M_t(v)| + \sum_{s=\tau_0}^{t-1} \frac{|E_s|}{2(s-1)}.
\end{eqnarray*}

For $\tau_0 \le s < T_v$, $X_{s+1}(v) - X_s(v)$ is either~0 or~1, and the probability that
it is equal to~1, conditional on $X_s$, is
$$
\frac{X_s(v)}{2(s-1)} \le 30 \omega^3 m_0 \sqrt{\frac{1}{(s-1)(\tau_0 -1)}}.
$$
So, for $t < T_v$, we have
\begin{eqnarray*}
\Phi_t^{f_v}(X) &=& \sum_{s=\tau_0}^t \sum_{x'} P_s(X_s,x') \left( f_v(x')-f_v(X_s) \right)^2 \\
&\le& \sum_{s=\tau_0}^t  30 \omega^3 m_0 \sqrt{\frac{1}{(s-1)(\tau_0 -1)}} \\
&\le& 60 \omega^3 m_0 \sqrt{\frac{t}{\tau_0-1}}.
\end{eqnarray*}

\smallskip

We now apply Theorem~\ref{thm.sparse}(b) to the function $f_v$, with $R(s) = 60 \omega^3 m_0 \sqrt{\frac{s}{\tau_0-1}}$,
$J = 1$, and with $\psi = \omega$.  Thus we have
\begin{eqnarray*}
\widetilde\delta(t) &=& 2 \max\left(\omega \log \omega,
\sqrt{60\omega^4 m_0 \log\left(60 \omega^3 m_0 \sqrt{(t-1)/(\tau_0-1)}\right)} \left(\frac{t-1}{\tau_0-1}\right)^{1/4} \right) \\
&\le& 8 \omega^2 \sqrt{m_0} \Big ( \frac{t-1}{\tau_0-1} \Big )^{1/4} \sqrt{\log \left(60 \omega^3 m_0 \sqrt{(t-1)/(\tau_0-1)}\right)},
\end{eqnarray*}
for $\tau_0 \le t$.  The result implies that, with probability at least
$1- 5e^{-\omega/4}$, we have
$|M_t(v)| \le \widetilde\delta(t)$ for all $t$ with $\tau_0 \le t \le T_v$.

We thus obtain that, with probability at least $1-5e^{-\omega/4}$, for $\tau_0 \le t \le T_v$:
$$
|E_t| \le \sum_{s=\tau_0}^{t-1} \frac{|E_s|}{2(s-1)} +
8 \omega^2 \sqrt{m_0} \left( \frac{t-1}{\tau_0-1} \right)^{1/4} \sqrt{ \log \left(60 \omega^3 m_0 \sqrt{(t-1)/(\tau_0-1)}\right)}.
$$
Now $\log (xy) - y^{1/3}\log x$ is decreasing in $y$ for $y \ge 1$ when $x \ge 20$, and is zero at
$y=1$: we apply this with $x=60 \omega^3 m_0$ and $y = \sqrt{(t-1)/(\tau_0-1)}$, to obtain that
\begin{eqnarray*}
\log \left(60 \omega^3 m_0 \sqrt{(t-1)/(\tau_0-1)}\right) &\le& \left( \frac{t-1}{\tau_0-1} \right)^{1/6} \log (60 \omega^3 m_0) \\
&\le& \omega^2 \left( \frac{t-1}{\tau_0-1} \right)^{1/6} \log (2m_0),
\end{eqnarray*}
for any $m_0 \ge 1$, and so
$$
|E_t| \le \sum_{s=\tau_0}^{t-1} \frac{|E_s|}{2(s-1)} +
8 \omega^3 \sqrt{m_0 \log (2 m_0)} \left( \frac {t-1} {\tau_0-1} \right)^{1/3}.
$$

We analyse the recurrence above using the following simple lemma.

\begin{lemma}
Let $A$ be a positive constant, and $\tau_0$ a positive
integer.  Suppose the sequence $e_t$, for $t \ge \tau_0$, satisfies $e_{\tau_0} = 0$ and
$$
e_t \le \sum_{s=\tau_0}^{t-1} \frac{e_s}{2(s-1)} + A \left( \frac{t-1}{\tau_0-1} \right)^{1/3},
$$
for all $t > \tau_0$.
Then
$$
e_t \le e^*_t = 6 A \left[ \sqrt{\frac{t-1}{\tau_0-1}}
- \frac{2}{3} \left( \frac{t-1}{\tau_0-1} \right)^{1/3} \right ],
$$
for all $t \ge \tau_0$.
\end{lemma}

Of course, the conclusion that we shall use is that
$e_t < 6 A \sqrt{\frac{t-1}{\tau_0 -1}}$, but the bound above is easier to
establish by induction.

\begin{proof}
The proof is by induction on $t$, the result being true with something to
spare for $t=\tau_0$.

Suppose the result is true for all $s$ with $\tau_0 \le s < t$.  Then, by the
induction hypothesis and the recursive bound, we have:
\begin{eqnarray*}
e_t &\le& \sum_{s=\tau_0}^{t-1} \frac{e^*_s}{2(s-1)}
+ A \left( \frac{t-1}{\tau_0-1} \right)^{1/3} \\
&\le& 3 A \sum_{s=\tau_0}^{t-1} \frac{1}{s-1}
\left[ \sqrt{\frac{s-1}{\tau_0-1}} -
\frac{2}{3} \left( \frac{s-1}{\tau_0-1} \right)^{1/3} \right ]
+ A \left( \frac{t-1}{\tau_0-1} \right)^{1/3}.
\end{eqnarray*}

Now the function
$g(s) = \frac{1}{s-1} \left[ \sqrt{\frac{s-1}{\tau_0-1}} - \frac{2}{3} \left(\frac{s-1}{\tau_0-1}\right)^{1/3} \right]$
is decreasing for all $s > \tau_0$, so $g(s) \le 1/3\tau_0$ for all
$s \ge \tau_0$, and we have
\begin{eqnarray*}
\sum_{s=\tau_0}^{t-1} g(s) &\le& \int_{s=\tau_0}^t g(s) \, ds + \frac{1}{3(\tau_0-1)} \\
&=& \frac{1}{3(\tau_0-1)} + \left[ 2 \left(\frac{s-1}{\tau_0-1}\right)^{1/2}
- 2 \left(\frac{s-1}{\tau_0-1}\right)^{1/3} \right]_{\tau_0}^t \\
&=& \frac{1}{3(\tau_0-1)} + 2 \left[ \left( \frac{t-1}{\tau_0-1} \right)^{1/2}
- \left( \frac{t-1}{\tau_0-1} \right)^{1/3} \right ].
\end{eqnarray*}

This gives
\begin{eqnarray*}
e_t &\le& \frac{A}{(\tau_0-1)} + 6 A \left( \frac{t-1}{\tau_0-1} \right)^{1/2}
- 5 A \left( \frac{t-1}{\tau_0-1} \right)^{1/3} \\
&\le& 6 A \left( \frac{t-1}{\tau_0-1} \right)^{1/2}
- 4 A \left( \frac{t-1}{\tau_0-1} \right)^{1/3}.
\end{eqnarray*}
This is the desired inequality for $e_t$.
\end{proof}

We can now deduce that, with probability at least $1-5e^{-\omega/4}$, for $\tau_0 \le t \le T_v$,
\begin{equation} \label{et}
|E_t| < 48 \omega^3 \sqrt {m_0  \log (2m_0)} \sqrt{\frac{t-1}{\tau_0 -1}},
\end{equation}
and this bound is at most $\displaystyle 48\omega^3 m_0 \sqrt{\frac{t-1}{\tau_0 -1}}$.
Combined with (\ref{xt}), this implies that $\displaystyle X_t(v) \le 50 \omega^3 m_0\sqrt{\frac{t-1}{\tau_0 -1}}$ for all times $t \le T_v$.
We deduce that $T_v = \infty$, since otherwise this would contradict the definition of $T_v$.
This means that, with probability at least $1-5e^{-\omega/4}$, the bound (\ref{et}) is valid for all times $t\ge \tau_0$.

We thus have the following theorem.

\begin{theorem}\label{thm:early-degrees}
For all $\omega \ge 4$, $\tau_0 \ge 4$, and $m_0 \ge 1$, we have
$$
\P \left( |X_t(v) - x_t| < 48 \omega^3 \sqrt {m_0 \log (2m_0)} \sqrt{\frac{t-1}{\tau_0-1}} \mbox{ for all $t \ge \tau_0$}
\,\Big|\, X_{\tau_0}(v) = m_0 \right)
$$
$$
\ge 1 - e^{-5\omega/4},
$$
and therefore
$$
\P \left( X_t(v) \le 50 \omega^3 m_0 \sqrt{\frac{t-1}{\tau_0-1}} \mbox{ for all $t \ge \tau_0$}
\,\Big|\, X_{\tau_0}(v) = m_0 \right)
$$
$$
\ge 1 - e^{-5\omega/4}.
$$
\end{theorem}

We note two consequences of the result above that we shall use later.

\begin{corollary} \label{cor:2times}
For $\tau_0 \ge 4$, $\omega \ge 4$ and $m_0 \ge 10^5 \omega^7$, we have
$$
\P \left( X_t(v) \le 2 m_0 \sqrt{\frac{t-1}{\tau_0-1}} \mbox{ for all $t \ge \tau_0$} \,\Big|\, X_{\tau_0}(v) = m_0 \right)
\ge 1 - e^{-5\omega/4}.
$$
\end{corollary}

\begin{proof}
By (\ref{xt}), we have $x_t \le \frac{5}{4} m_0 \sqrt{\frac{t-1}{\tau_0-1}}$ for all $t \ge \tau_0$, given the initial condition $x_{\tau_0} = m_0$.

The result will then follow from Theorem~\ref{thm:early-degrees} as long as
$$
48 \omega^3 \sqrt {\frac{\log (2m_0)}{m_0}} \le \frac34.
$$
We see that
$$
\frac{m_0}{\log(2m_0)} \ge \frac{10^5 \omega^7}{\log(2\times 10^5 \omega^7} \ge \omega^6 \frac{10^5 \times 4}{\log(2 \times 10^5 \times 4^7)}
\ge \omega^6 2^{12},
$$
which implies the desired inequality.
\end{proof}

We shall also use the following result, stating that the maximum degree at time $t$ is unlikely to be larger than
$\psi \sqrt{t-1}$, where $\psi$ is a large constant.

\begin{theorem} \label{maxdegree}
Let $\tau_0 \ge 4$ and $\psi \ge 10^5 \sqrt{\tau_0-1} \log^3\tau_0$ be constants.
For the preferential attachment model, with any initial graph on $\tau_0$ vertices
and $\tau_0-1$ edges,
$$
\P\left( X_t(v) > \psi \sqrt{t-1} \mbox{ for some vertex $v$ and some $t \ge \tau_0$}\right)
$$
$$
\le 2 \tau_0 \exp\left( - \frac{\psi^{1/3}}{3(\tau_0-1)^{1/6}} \right) \le \frac{1}{\psi}.
$$
\end{theorem}

\begin{proof}
Let $P_1$ be the probability that $X_t(v) \ge \psi \sqrt{t-1}$ for some $t \ge \tau_0$ and some vertex $v$
already present at time $\tau_0$, and $P_2$ be the probability that $X_t(v) \ge \psi \sqrt{t-1}$ for some
$t \ge \tau_0$ and some vertex $v$ arriving at a time later than $\tau_0$.

We begin by bounding $P_1$.
For a fixed vertex $v$ present at time $\tau_0$, its degree at that time is certainly at most $\tau_0-1$.
We apply Theorem~\ref{thm:early-degrees} with $m_0 = \tau_0 -1$, and $\omega = (\psi/50)^{1/3} (\tau_0-1)^{-1/6} \ge 4$,
so that
$$
50 \omega^3 m_0 \sqrt{\frac{t-1}{\tau_0-1}} \le 50 \omega^3 \sqrt{\tau_0-1} \sqrt{t-1}= \psi \sqrt{t-1}.
$$
We obtain that
$$
\P \left( X_t(v) > \psi \sqrt{t-1} \mbox{ for some $t \ge \tau_0$}\right) \le e^{-5\omega/4}
\le \exp\left( - \frac{\psi^{1/3}}{3(\tau_0-1)^{1/6}} \right).
$$
We therefore have
$$
P_1 \le \tau_0 \exp\left( - \frac{\psi^{1/3}}{3(\tau_0-1)^{1/6}} \right)
$$

\smallskip

We now bound $P_2$.
For each time $s > \tau_0$, consider the new vertex $v$ of degree~1 born at time~$s$.  We apply Theorem~\ref{thm:early-degrees} to
this vertex, with $m_0 =1$, $\tau_0$ replaced by $s$, and $\omega = \left( \psi \sqrt{s-1}/50 \right)^{1/3}$, so
$50 \omega^3 m_0 \sqrt{\frac{t-1}{s-1}} \le \psi \sqrt{t-1}$, and we have
$$
\P \left( X_t(v) > \psi \sqrt{t-1} \mbox{ for some $t \ge s$}\right) \le e^{-5\omega/4}
\le \exp\left( - \frac{\psi^{1/3} (s-1)^{1/6}}{3}\right).
$$
Summing over $s$, we have
$$
P_2 \le \sum_{s=2}^\infty \exp\left( - \frac{\psi^{1/3} (s-1)^{1/6}}{3}\right) \le
2 \exp\left( - \psi^{1/3} / 3 \right).
$$

\smallskip

The overall probability that there is, at any time $t$, a vertex of degree at least $\psi \sqrt{t-1}$, is thus
at most $P_1 + P_2 \le 2P_1$, as claimed.

For the final inequality, we need to show that our bounds imply that
$$
f(\psi) = \psi^{1/3} - 3 (\tau_0-1)^{1/6} \log (2 \tau_0 \psi) \ge 0.
$$
We note that $f'(\psi) = \frac{1}{3\psi} \left[ \psi^{1/3} - 9 (\tau_0-1)^{1/6} \right] > 0$, so it is enough
to verify that the inequality holds at $\psi = 10^5 \sqrt{\tau_0-1} \log^3\tau_0$, at which point $2\psi \le \tau_0^{12}$ for all
$\tau_0 \ge 4$.  The desired inequality is equivalent to $10^{5/3} \log (\tau_0) \ge 3 \log (2 \tau_0 \psi)$: this holds since $10^{5/3} > 39$.
\end{proof}

\section{Concentration for $D_t(\ell)$} \label{Ssimple}

In this section, we again consider the basic preferential attachment model, but now we are concerned with the number of vertices
of degree exactly $\ell$ at time $t$.

Recall that, for $t \ge \tau_0$ and $\ell \in \N$, $D_t(\ell)$ denotes the number of vertices of degree exactly $\ell$ at time $t$.
It is easy to see that $D = (D_t (\ell): t \ge \tau_0, \ell \in \N)$ is a Markov chain.

We recall our main theorem.

\begin{reptheorem} {thm:main}
Let $\tau_0 \ge 4$ and $\psi \ge 10^5 \sqrt{\tau_0-1} \log^3\tau_0$ be constants.
Let $G(\tau_0)$ be any graph with $\tau_0$ vertices and $\tau_0 -1$ edges, and consider the
preferential attachment process with initial graph $G(\tau_0)$ at time $\tau_0$, and the associated Markov chain
$D = (D_t (\ell): t \ge \tau_0, \ell \in \N)$.

With probability at least $1 - \frac{4}{\psi}$, we have
$$
\left| D_t(\ell) - \frac{4t}{\ell(\ell+1)(\ell+2)} \right| \le 120 \sqrt{\frac{t \log (\psi t)}{\ell^3}} + 301 \psi^2 \log(\psi t),
$$
for all $\ell \ge 1$, and all $t \ge \tau_0$.
\end{reptheorem}

As is well-known (and as we shall show shortly), the expectation of $D_t(\ell)$ is very close to $4t/\ell(\ell+1)(\ell+2)$, for all $\ell \ge 1$
and all $t \ge \tau_0$, so the theorem shows concentration of measure of these random variables about their means.

For $\ell \le (t / \log (\psi t))^{1/3}/\psi^2$, the bound on the deviation of $D_t(\ell)$ is at most
$\displaystyle 125 \sqrt{\frac{t \log (\psi t)}{\ell^3}}$, which is, up to the log factor, on the order of $\sqrt{\E D_t(\ell)}$.
We get concentration within a factor $(1+o(1))$ of the mean as long as $\ell = o(t / \log t)^{1/3}$.

\smallskip

For all values of $\ell$ larger than $(t / \log t)^{1/3}$, the bound on the deviation that we obtain is of
order $\log t$.  This result might conceivably be of interest for values of $\ell$ between about $(t / \log t)^{1/3}$ and $t^{1/2}$,
but for larger values of $\ell$ we already have a stronger result: Theorem~\ref{maxdegree} tells us that
$D_t(\ell)=0$ when $\ell$ is larger than $\psi \sqrt t$, with probability at least $1-1/\psi$.

\smallskip

The proof of Theorem~\ref{thm:main} takes up the rest of this section, although we defer the bulk of the
calculations until later sections.

\begin{proof}
It will shortly turn out to be convenient to truncate the range of $\ell$, so that we consider only values of $\ell$ with
$1 \le \ell \le \ell_0$, for some fixed $\ell_0$.  We remark now that we may freely do this, as we are proving an explicit bound
on the probability of failure that is independent of $\ell_0$.

\smallskip

For the moment though, we consider all values $\ell \in \N$ simultaneously, and consider the evolution of the entire process
$D = (D_t(\ell))$ for $t \ge \tau_0$.

We have, for $t \ge \tau_0+1$,
\begin{eqnarray*}
\E [D_t(1)-D_{t-1} (1) \mid D_{t-1}] = 1 - \frac{D_{t-1} (1)}{2(t-1)}.
\end{eqnarray*}
Also, for $\ell \ge 2$,
\begin{eqnarray*}
\E [D_t (\ell) - D_{t-1} (\ell ) \mid D_{t-1}] = \frac{(\ell-1) D_{t-1} (\ell-1)}{2(t-1)} - \frac{\ell  D_{t-1} (\ell)}{2(t-1)}.
\end{eqnarray*}
Then, by Lemma~\ref{lem.mart},
\begin{eqnarray*}
D_t (1) & = & D_{\tau_0} (1) + \sum_{s=\tau_0}^{t-1} \Big ( 1 - \frac{D_s(1)}{2s} \Big ) + M_t (1)\\
D_t (\ell) & = & D_{\tau_0} (\ell) + \sum_{s = \tau_0}^{t-1} \Big ( \frac{(\ell-1)D_s (\ell-1)}{2s} - \frac{\ell D_s (\ell)}{2s} \Big )
+ M_t (\ell), \quad (\ell \ge 2),
\end{eqnarray*}
where $M_t (\ell)$ is a martingale for each $\ell \ge 1$.

We want to show that, for $\ell \ge 1$, $D_t (\ell)$ is close to $d_t (\ell)$, where the $d_t(\ell)$ satisfy
$d_{\tau_0}(\ell) = D_{\tau_0}(\ell)$ for all $\ell$, and:
\begin{eqnarray*}
d_t (1) & = & d_{\tau_0} (1) + \sum_{s=\tau_0}^{t-1} \Big ( 1 - \frac{d_s(1)}{2s} \Big )\\
d_t (\ell) & = & d_{\tau_0} (\ell) + \sum_{s = \tau_0}^{t-1}
\Big ( \frac{(\ell-1)d_s (\ell-1)}{2s} - \frac{\ell d_s (\ell)}{2s} \Big ), \quad (\ell \ge 2).
\end{eqnarray*}
Given the initial values $d_{\tau_0}(\ell) = D_{\tau_0}(\ell)$, for $\ell \ge 1$, the equations above are equivalent to:
\begin{eqnarray*}
d_t (1) & = & 1 + d_{t-1}(1) \Big ( 1 - \frac{1}{2(t-1)} \Big )\\
d_t (\ell) & = & d_{t-1} (\ell) \Big ( 1- \frac{\ell}{2(t-1)} \Big ) + \frac{\ell-1}{2(t-1)} d_{t-1} (\ell-1), \quad (\ell \ge 2).
\end{eqnarray*}
These equations are known to admit the explicit solution
$$
d_t(\ell) = \frac{4t}{\ell (\ell+1)(\ell+2)},
$$
if the initial conditions correspond (which of course cannot happen for a concrete graph at time $\tau_0$, since then
all the $D_{\tau_0}(\ell)$ are natural numbers).
More generally, we have the following result, which is very similar to results of Szyma\'nski~\cite{Szy1,Szy2} and
Bollob\'as, Riordan, Spencer and Tusn\'ady~\cite{brst01}.

\begin{lemma} \label{lem:dt}
Take any $\tau_0 \ge 1$, and any sequence $(D_{\tau_0}(\ell))_{\ell \ge 1}$ of non-negative integers with
$\sum_{\ell \ge 1} D_{\tau_0}(\ell) = \tau_0$.  Then the solution $d_t(\ell)$ of the equations above
with $d_{\tau_0}(\ell) = D_{\tau_0}(\ell)$ for all $\ell$ satisfies
$$
\left| d_t(\ell) - \frac {4t}{\ell(\ell+1)(\ell+2)} \right| \le \frac{\tau_0^{3/2}}{t^{1/2}} \le \tau_0,
$$
for all $\ell \ge 1$ and all $t \ge \tau_0$.
\end{lemma}

\begin{proof}
Set
$$
z_t(\ell) = d_t(\ell) - \frac{4t}{\ell(\ell+1)(\ell+2)},
$$
for all $t \ge \tau_0$ and $\ell \ge 1$.

Note first that $D_{\tau_0}(\ell) \le \tau_0$, for all $\ell$, and so also $|z_{\tau_0}(\ell)| \le \tau_0$.
Thus the lemma holds for $t = \tau_0$.

For each $t > \tau_0$, it is straightforward to verify that
$$
z_t(1) = z_{t-1}(1) \left( 1 - \frac{1}{2(t-1)} \right),
$$
%\begin{eqnarray*}
%z_t(1) &=& d_t(1) - \frac{2t}{3} \\
%&=& 1 + \left(z_{t-1}(1) + \frac{2(t-1)}{3} \right) \left( 1 - \frac{1}{2(t-1)} \right) - \frac{2t}{3} \\
%&=& z_{t-1}(1) \left( 1 - \frac{1}{2(t-1)} \right) + 1 + \frac{2(t-1)}{3} - \frac13 -\frac{2t}{3} \\
%&=& z_{t-1}(1) \left( 1 - \frac{1}{2(t-1)} \right),
%\end{eqnarray*}
and, for $\ell >1$,
$$
z_t(\ell) = z_{t-1}(\ell) \left( 1 - \frac{\ell}{2(t-1)} \right) + z_{t-1}(\ell-1) \frac{\ell-1}{2(t-1)}.
$$
%\begin{eqnarray*}
%z_t(\ell) &=& d_t(\ell) - \frac{4t}{\ell(\ell+1)(\ell+2)} \\
%&=& \left( z_{t-1}(\ell) + \frac{4(t-1)}{\ell(\ell+1)(\ell+2)} \right) \left( 1 - \frac{\ell}{2(t-1)} \right) \\
%&& + \left(z_{t-1}(\ell-1) + \frac{4(t-1)}{(\ell-1)\ell(\ell+1)} \right) \frac{\ell-1}{2(t-1)} - \frac{4t}{\ell(\ell+1)(\ell+2)} \\
%&=& z_{t-1}(\ell) \left( 1 - \frac{\ell}{2(t-1)} \right) + z_{t-1}(\ell-1) \frac{\ell-1}{2(t-1)} \\
%&& \mbox{} + \frac{4(t-1) -2\ell +2(\ell+2) -4t}{\ell(\ell+1)(\ell+2)} \\
%&=&  z_{t-1}(\ell) \left( 1 - \frac{\ell}{2(t-1)} \right) + z_{t-1}(\ell-1) \frac{\ell-1}{2(t-1)}.
%\end{eqnarray*}
If $Z$ is a common upper bound on $|z_{t-1}(\ell)|$ and $|z_{t-1}(\ell-1)|$, we deduce that
$$
|z_t(\ell)| \le Z\left( 1 - \frac{\ell}{2(t-1)} + \frac{\ell-1}{2(t-1)} \right) = Z \left( 1 - \frac{1}{2(t-1)} \right).
$$
By induction, it now follows that
$$
|z_t(\ell)| \le \tau_0 \prod_{u=\tau_0}^{t-1} \left(1-\frac{1}{2u} \right) \le
\tau_0 \sqrt{\frac{\tau_0}{t}} = \frac{\tau_0^{3/2}}{t^{1/2}},
$$
for all $t \ge \tau_0$ and every $\ell\ge 1$, as claimed.
\end{proof}

\smallskip

Define $E_t (\ell) = D_t (\ell) - d_t(\ell)$, where, as above, we set $d_{\tau_0}(\ell) = D_{\tau_0}(\ell)$ for all $\ell \ge 1$.  Note that
$D_{\tau_0}(\ell)$ is an integer-valued random variable, determined by the graph at the initial time $\tau_0$.  Note also that
$E_{\tau_0}(\ell) = 0$ for all $\ell \ge 1$.  For the moment, we shall keep the term $E_{\tau_0}$ in our expressions, to show how the
calculation would be affected in a setting where $E_{\tau_0}$ is not necessarily zero.

For $t \ge \tau_0+1$, we have
\begin{eqnarray*}
E_t(1) & = & E_{\tau_0} (1) - \sum_{s=\tau_0}^{t-1} \frac{E_s(1)}{2s} + M_t(1)\\
E_t (\ell) & = & E_{\tau_0} (\ell) + \sum_{s=\tau_0}^{t-1} \Big (\frac{(\ell-1)E_s(\ell-1)}{2s} - \frac{\ell E_s (\ell)}{2s} \Big ) + M_t (\ell),
\quad (\ell \ge 2).
\end{eqnarray*}
This means that, for $t \ge \tau_0+1$,
\begin{eqnarray*}
E_t (1) & = & E_{t-1} (1) \Big (1 - \frac{1}{2(t-1)} \Big ) + M_t (1) - M_{t-1}(1)
\end{eqnarray*}
and, for $\ell \ge 2$,
\begin{eqnarray*}
E_t (\ell) & = & E_{t-1} (\ell) \Big (1 - \frac{\ell}{2(t-1)} \Big ) + \frac{(\ell-1) E_{t-1} (\ell-1)}{2(t-1)} + M_t (\ell) - M_{t-1} (\ell).
\end{eqnarray*}

At this point, we truncate the process $D$: we fix some $\ell_0 \ge 1$, and set
$D^{\ell_0}= (D_t^{\ell_0} (\ell): t \in \Z^+, \ell =1, \ldots, \ell_0)$ -- in other words, we restrict attention to the
numbers of vertices with degrees at most $\ell_0$.  The truncated process $D^{\ell_0}$ remains Markov, since the distribution of
$D_t(\ell)$ conditioned on $D_{t-1}$ depends only on $D_{t-1}(\ell)$ and $D_{t-1}(\ell-1)$, for each $\ell \le \ell_0$.

Once we have fixed $\ell_0$, we may restate the previous system of equations as a matrix equation, giving a recurrence for
\begin{eqnarray*}
E_t=
\begin{pmatrix}
E_t(1)\\
\vdots \\
E_t (\ell_0)
\end{pmatrix}.
\end{eqnarray*}
We have, for $t \ge \tau_0 + 1$,
\begin{eqnarray*}
E_t = A_{t-1} E_{t-1} + \Delta M_t,
\end{eqnarray*}
where
\begin{eqnarray*}
\Delta M_t=
\begin{pmatrix}
\Delta M_t(1)\\
\vdots \\
\Delta M_t (\ell_0)
\end{pmatrix},
\end{eqnarray*}
with $\Delta M_t (\ell) = M_t (\ell) - M_{t-1} (\ell)$ for each $\ell$, and, for $s \ge \tau_0$, the matrix $A_s$ is given by
\begin{eqnarray*}
\begin{pmatrix}
1-\frac{1}{2s} & 0 & \cdots & 0\\
\frac{1}{2s} & 1-\frac{2}{2s} &  \cdots & 0\\
\vdots & \ddots & \ddots & 0\\
0 & 0 & \frac{\ell_0-1}{2s} & 1-\frac{\ell_0}{2s}
\end{pmatrix}.
\end{eqnarray*}

Hence it follows that, for $t \ge \tau_0+1$,
\begin{eqnarray*}
E_t = \left(\prod_{s=\tau_0}^{t-1} A_s \right) E_{\tau_0} + \sum_{s=\tau_0+1}^t \left(\prod_{u=s}^{t-1} A_u\right) \Delta M_s.
\end{eqnarray*}
Here and subsequently, the notation $\prod_{s=\tau_0}^{t-1} A_s$ indicates the matrix product $A_{t-1}\cdots A_{\tau_0}$, with the
indices taken in decreasing order.
At this point, we recall that $E_{\tau_0} = 0$, so that
$$
E_t = \sum_{s=\tau_0+1}^t \left(\prod_{u=s}^{t-1} A_u\right) \Delta M_s.
$$

We shall control the deviations of $E_t$, although this process is not itself a martingale, and so we cannot directly apply our
martingale deviation inequalities.  The process $(E_t)$ is a transform of the martingale $(M_t)$, in that it is a sum of its differences,
multiplied by the appropriate $\prod_{u=s}^{t-1} A_u$, which depend on $t$.  In order to get around this difficulty, we now introduce,
for each $\tau > \tau_0$, a martingale $\widetilde{M}^\tau$ stopped at $\tau$, whose value at $\tau$ is the quantity $E_\tau$ of interest.

We fix $\tau > \tau_0$ and define, for $t \le \tau$,
$$
\widetilde{M}^\tau_t = \sum_{s=\tau_0+1}^t \left(\prod_{u=s}^{\tau-1} A_u\right) \Delta M_s,
$$
and $\widetilde{M}^\tau_t = \widetilde{M}^\tau_{\tau}$ for $t > \tau$, then it is easily checked that
$\widetilde{M}^\tau = (\widetilde{M}^\tau_t)$ is a martingale,
and that $\widetilde{M}^\tau_{\tau} = E_{\tau}$. Thus we can obtain bounds on $E_{\tau}$ by studying the martingale $\widetilde{M}^\tau$.

\smallskip

From Lemma~\ref{lem:dt}, we have that, for every $\tau \ge \tau_0$ and every $\ell = 1,\dots, \ell_0$,
$$
D_{\tau}(\ell) = d_{\tau}(\ell) + \widetilde{M}^\tau_{\tau} (\ell) \le \frac{4\tau}{\ell^3} + \tau_0 + \widetilde{M}^\tau_{\tau} (\ell).
$$

\smallskip

We now consider the transitions of the truncated process $D^{\ell_0}$, with state space $(\Z^+)^{\ell_0}$.  Recall that each
transition involves the creation of one new vertex of degree~1, and the increase of a degree of an existing vertex by~1.  This means
that a transition of the truncated process involves an increase of~1 in $D^{\ell_0}(1)$, and either: (i)~a decrease of~1
in $D^{\ell_0}(k)$ and an increase of~1 in $D^{\ell_0}(k+1)$, for some $k \in \{ 1, \dots, \ell_0-1\}$, (ii)~a decrease of~1 in
$D^{\ell_0}(\ell_0)$, or (iii)~no further change.  In other words, the vector $D_{s+1}^{\ell_0}$ is obtained from
$D_s^{\ell_0}$ by adding one of the following vectors: \\
(i)~$y_k=e_1-e_k+e_{k+1}$, for some $k \in \{1, \dots, \ell_0-1\}$, \\
(ii)~$y_{\ell_0} = e_1 - e_{\ell_0}$, \\
(iii)~$y_0=e_1$.  \\
Here $e_j$ denotes the standard basis vector in $\Z^{\ell_0}$ with a~1 in the
$j$th coordinate and 0s elsewhere: here and in what follows, we abuse notation by suppressing the dependence on $\ell_0$.
The transition probabilities are then given by
\begin{eqnarray*}
P_s(D_s, D_s + y_k) &=& \frac{k D_s(k)}{2s-2}  \qquad (k=1, \dots, \ell_0) \\
P_s(D_s, D_s + y_0) &=& 1 - \sum_{k=1}^{\ell_0} \frac{k D_s(k)}{2s-2}\,.
\end{eqnarray*}
Here too we have removed the superscripts $\ell_0$ for clarity.

\smallskip

We can write, for $s \ge \tau_0$,
\begin{eqnarray*}
\Delta M_{s+1} &=& D_{s+1} - D_s - \sum_{k=0}^{\ell_0} P_s(D_s,D_s+y_k) y_k \\
&=& \sum_{k=0}^{\ell_0} y_k [\1_{D_{s+1}-D_s=y_k}-P_s(D_s,D_s+y_k)].
\end{eqnarray*}

We consider running the process up to some fixed $\tau > \tau_0$: all our notation should specify the dependence on $\tau$, but
again where possible we shall suppress this.

For $\tau_0 \le s < \tau$, we define $B_s = B_s^\tau = \prod_{u=s+1}^{\tau-1} A_u$, so that
$$
\widetilde{M}_s^\tau = \sum_{w=\tau_0}^{s-1} B_w \Delta M_{w+1}.
$$
We then have, for $\tau_0 \le s < \tau$,
$$
\Delta \widetilde{M}^\tau_{s+1} = \widetilde{M}^\tau_{s+1}-\widetilde{M}^\tau_s = B_s \Delta M_{s+1} =
B_s\sum_{k=0}^{\ell_0} y_k [\1_{D_{s+1}-D_s=y_k}-P_s(D_s,D_s+y_k)].
$$

Now we define $\widetilde{D} = \widetilde{D}^{\ell_0}$ by
\begin{eqnarray*}
\widetilde{D}_t & = & \sum_{s=\tau_0}^{t-1} B_s\sum_{k=0}^{\ell_0} y_k P_s (D_s,D_s+y_k)\\
&&\mbox{} + \sum_{s=\tau_0}^{t-1} B_s\sum_{k=0}^{\ell_0} y_k [\1_{D_{s+1}-D_s=y_k}-P_s(D_s,D_s+y_k)], \quad (t \le \tau), \\
&=& \sum_{s=\tau_0}^{t-1} \sum_{k=0}^{\ell_0} P_s(D_s,D_s+y_k) [B_s y_k] + \widetilde{M}_t^\tau.
\end{eqnarray*}
so that, for $t \le \tau$,
\begin{eqnarray*}
\widetilde{D}_t & = & \sum_{s=\tau_0}^{t-1} B_s\sum_{k=0}^{\ell_0} y_k \1_{D_{s+1}-D_s=y_k} \\
&=& \sum_{s=\tau_0}^{t-1} B_s (D_{s+1}-D_s)
\end{eqnarray*}
and so
$$
\widetilde{D}_{t+1} - \widetilde{D}_t = B_t(D_{t+1} - D_t).
$$

\smallskip

The process $\widetilde{D}$ is not in general a Markov process.  However, we may define a process $Y = Y^\tau$ by setting
$Y_t = (D_t, \widetilde{D}_t)$ for $t \le \tau$, and $Y_t = Y_{\tau}$ for $t \ge \tau$.  This extended process $Y$ is Markovian,
with state space $E = (\Z^+)^{\ell_0} \times (\R^+)^{\ell_0}$.
At each time $t$ with $\tau_0 \le t < \tau$, the one-step transition matrix $\widetilde{P}_t$ for $Y$ is derived from that of $D$.
Specifically, if $D_{t+1}-D_t = y_k$, then $\widetilde{D}_{t+1} - \widetilde{D}_t = B_t y_k$, and $Y_{t+1} - Y_t = (y_k, B_ty_k)$.

Our plan is to apply Theorem~\ref{thm.key} to the Markov process $Y$, and, for $\ell = 1, \dots, \ell_0$, to the projection function
$g=g^\ell$ taking $(x,\widetilde{x}) \in E$ to $\widetilde{x}(\ell)$.  For each $k = 0, \dots, \ell_0$, if $D_{t+1} - D_t = y_k$,
then $g(Y_{t+1}) - g(Y_t) = [B_t y_k](\ell)$, the $\ell$-th entry of the vector $B_t y_k$.  Since $B_t$ is a product of non-negative
sub-stochastic matrices, it too is
non-negative and sub-stochastic.  The vector $y_k$ has all its entries in $\{ 0, +1, -1\}$, with at most two positive and one
negative entries, so each co-ordinate of the
vector $B_t y_k$ is a sum of at most two entries of $B_t$, minus at most one other entry.  Therefore $|[B_t y_k] (\ell)| \le 1$ for
all $t$, $k$ and $\ell$.  So, in applying Theorem~\ref{thm.key}, we may take $J=1$.

For $\ell = 1,\dots, \ell_0$, and $\tau_0 \le t < \tau$, we therefore have
\begin{eqnarray*}
\Phi^{g^\ell}_t(Y) &=& \sum_{s=\tau_0}^t \sum_{x'} P_s(D_s,x') \left( g^\ell(x') - g^\ell(D_s) \right)^2 \\
&=& \sum_{s=\tau_0}^t \sum_{k=0}^{\ell_0} P_s(D_s,D_s+y_k) \left( [B_s y_k](\ell) \right)^2.
\end{eqnarray*}
For brevity, we set $\Phi^\ell_t(Y) = \Phi^{g^\ell}_t(Y)$ from now on.

For each $\ell = 1, \dots, \ell_0$, we set
$$
R^\ell = 1600 \frac{\tau-1}{\ell^3} + (10\psi)^4 \log(\psi \tau),
$$
and, for $\ell=1, \dots, \ell_0$, we set
$$
T_R^\ell = \inf \{ t \ge \tau_0 : \Phi^\ell_t(Y) > R^\ell \}.
$$
We now apply Theorem~\ref{thm.key}, with $J=1$, $R=R^\ell$, and $\omega = 9 \log (\psi \tau)$, noting that $\omega J^2\le R^\ell$;
we obtain that, for each $\tau \ge \tau_0$ and $\ell = 1, \dots, \ell_0$,
$$
\P \left( \left( \sup_{\tau_0\le t \le \tau} |\widetilde{M}_t^\tau(\ell)| > 3 \sqrt{\log(\psi\tau) R^\ell } \right)
\wedge (T_R^\ell \ge \tau) \right) \le \frac{2}{\psi^2 \tau^2}.
$$

\smallskip

Let $\widehat{T}_\Delta = \inf \{ s \ge \tau_0 : D_s(k) > 0 \mbox{ for some } k > \psi \sqrt {s-1} \}$, the first time $s$ such that
there is a vertex of degree greater than $\psi \sqrt{s-1}$: by Theorem~\ref{maxdegree},
$\P(\widehat{T}_\Delta < \infty) \le 1/\psi$.
Also, for each $k= 1, \dots, \ell_0$, let
$$
\widehat{T}_k = \inf \left\{ s \ge \tau_0 : D_s(k) > 5\frac{s}{k^3} + 400\psi^2 \log(\psi s) \right\}.
$$
Note that, for $\tau_0 \le s < \min \big( (k/\psi)^2+1 ,\widehat{T}_\Delta \big)$, $D_s(k) = 0$, and so
$\widehat{T}_k \ge \min\big( (k/\psi)^2 ,\widehat{T}_\Delta\big)$ for each $k = 1,\dots,\ell_0$.
Finally, let $\widehat{T}$ be the minimum of $\widehat{T}_1, \dots, \widehat{T}_{\ell_0}$.

\smallskip

In the next section, we shall prove the following result.

\begin{lemma}\label{PhivR}
For all $t \le \widehat{T} \land \widehat{T}_\Delta$, and all $\ell=1, \dots, \ell_0$, $\Phi_{t-1}^\ell(Y) \le R^\ell $.
\end{lemma}

This result can be restated as saying that $T_R^\ell \ge \widehat{T} \land \widehat{T}_\Delta$ for each $\ell = 1, \dots, \ell_0$.

\smallskip

Now, for each $\tau \ge \tau_0$ and each $\ell = 1, \dots, \ell_0$, set
$$
\delta_\tau(\ell) = 120 \sqrt{\frac{\tau \log (\psi \tau)}{\ell^3}} + 300 \psi^2 \log (\psi \tau) ,
$$
which is slightly less than the bound on the deviation appearing in the statement of the theorem.
Observe that, for $\tau\ge \tau_0$ and $\ell=1, \dots, \ell_0$,
$$
\delta_\tau(\ell)^2 \ge 14400 \frac{\tau \log (\psi \tau)}{\ell^3} + 9 \times (10\psi)^4 \log^2(\psi \tau) =
9 \log(\psi \tau) R^\ell.
$$
Thus we have
$$
\delta_\tau(\ell) \ge 3 \sqrt{\log(\psi \tau) R^\ell},
$$
and so, using Lemma~\ref{PhivR},
\begin{eqnarray*}
\lefteqn{\P \left( \left(\sup_{\tau_0 \le t \le \tau} |\widetilde{M}^\tau_t(\ell)| > \delta_\tau(\ell) \right)
\wedge (\widehat{T} \land \widehat{T}_\Delta \ge \tau) \right)} \\
&\le& \P \left( \left(\sup_{\tau_0 \le t \le \tau} |\widetilde{M}^\tau_t(\ell)| > \delta_\tau(\ell) \right)
\wedge (T^\ell_R \ge \tau) \right) \\
&\le& \frac{2}{\psi^2 \tau^2}.
\end{eqnarray*}
Recall that $\widetilde{M}^\tau_\tau = E_\tau$ for each $\tau \ge \tau_0$, so we now deduce that
\begin{equation} \label{failure}
\P \left( \left(|E_{\tau}(\ell)| > \delta_\tau(\ell)\right) \wedge (\widehat{T} \land \widehat{T}_\Delta\ge \tau) \right)
\le \frac{2}{\psi^2 \tau^2},
\end{equation}
for all $\tau \ge \tau_0$ and $\ell=1, \dots, \ell_0$.

\smallskip

We now wish to bound the total probability that there is some pair $(\tau, \ell)$, with $\tau \ge \tau_0$ and $\ell=1, \dots, \ell_0$,
such that $|E_{\tau}(\ell)| > \delta_\tau(\ell)$ and $\tau < \widehat{T} \land \widehat{T}_\Delta$: recall that we want a bound
independent of $\ell_0$.

For those pairs $(\tau, \ell)$ with $\ell \le \psi \sqrt{\tau-1}$, we sum the bounds from (\ref{failure}), and obtain that
\begin{eqnarray*}
\lefteqn{\P \Big( \left(|E_\tau(\ell)| > \delta_\tau(\ell) \right) \wedge (\widehat{T} \land \widehat{T}_\Delta > \tau)} \\
&& \qquad \mbox{ for some } \tau\ge \tau_0, 1\le \ell \le \min (\ell_0,\psi \sqrt {\tau-1}) \Big) \\
%$$
%\begin{eqnarray*}
&\le& \frac{2}{\psi^2} \sum_{\tau = \tau_0}^\infty \frac{\psi \sqrt{\tau-1} }{\tau^2} \\
&\le& \frac{2}{\psi} \int_3^\infty \tau^{-3/2} \, d\tau \\
&\le& \frac{3}{\psi}.
\end{eqnarray*}

For those pairs $(\tau, \ell)$ with $\ell > \psi \sqrt{\tau-1}$, we have either $\widehat{T}_\Delta \le \tau$ or $D_\tau(\ell) =0$, and in
the latter case we have
$$
|E_\tau(\ell)| = d_\tau(\ell) \le \frac{4\tau}{\ell^3} + \tau_0 < \frac{4}{\psi^2 \ell} + \frac{1}{\ell^3} + 10^4\psi^2 < \delta_\tau(\ell).
$$
Therefore, with probability at least $1- 3/\psi$, we have, for all
$\ell = 1, \dots, \ell_0$ and all $\tau \ge \tau_0$, that either $|E_{\tau}(\ell)| \le \delta_\tau(\ell)$ or $\tau \ge \widehat{T} \land \widehat{T}_\Delta$.

\smallskip

We now set $T^*_\ell = \inf \{ s \ge \tau_0 : |E_s(\ell)| > \delta_s(\ell)\}$ for each $\ell=1,\dots, \ell_0$, and
$T^* = \min (T^*_\ell, \ell=1, \dots, \ell_0)$;
we obtain that
\begin{equation} \label{times}
\P ( T^* < \infty \mbox{ and } T^* \le \widehat{T} \land \widehat{T}_\Delta) \le 3/\psi.
\end{equation}

On the other hand, if $\widehat{T}_\ell < T^*_\ell$ for some $\ell=1, \dots, \ell_0$, there
is an $s \ge \tau_0$ with $D_s(\ell) > 5 s/\ell^3 + 400 \psi^2 \log (\psi s)$ and
\begin{eqnarray*}
D_s(\ell) &\le& \frac{4s}{\ell^3} + \tau_0 + \delta_s(\ell) \\
&=& \frac{4s}{\ell^3} + \tau_0 + 120 \sqrt{\frac{s \log (\psi s)}{\ell^3}} + 300 \psi^2 \log (\psi s) \\
&\le& \frac{4s}{\ell^3} + \psi^2 + \left(\frac{s}{\ell^3} + 3600\log(\psi s)\right) + 300 \psi^2 \log(\psi s) \\
&\le& 5 \frac{s}{\ell^3} + 350 \psi^2 \log(\psi s),
\end{eqnarray*}
which is a contradiction.

We conclude that $T^*_\ell \le \widehat{T}_\ell$ for all $\ell=1, \dots, \ell_0$.
Recalling that $\widehat{T}$ is the minimum of $\widehat{T}_1, \dots, \widehat{T}_{\ell_0}$ and
$T^*$ is the minimum of $T^*_1, \dots, T^*_{\ell_0}$,
this implies that $T^* \le \widehat{T}$.
Equation~(\ref{times}) now implies that $\P (T^* < \infty \mbox{ and } T^* \le \widehat{T}_\Delta) \le 3/\psi$.
However, we also have that $\P(\widehat{T}_\Delta < \infty) \le 1/\psi$, so
$\P(T^* < \infty) \le 4/\psi$.

This conclusion is equivalent to the statement that
$$
\P\left( |E_s(\ell)| \le \delta_s(\ell) \mbox{ for all $\ell \ge 1$ and $s \ge \tau_0$} \right) \ge 1 - \frac{4}{\psi}.
$$
This implies the result stated, since
$$
\left| D_s(\ell) - \frac{4s}{\ell(\ell+1)(\ell+2)} \right| \le |E_s(\ell)| + \tau_0 \le |E_s(\ell)| + \psi^2.
$$
\end{proof}

\section{Bounds for $\Phi^\ell_{\tau-1}(Y)$}  \label{Sbounds}

Our aim in this section is to prove Lemma~\ref{PhivR}, which states that
$$
\Phi^\ell_{\tau-1}(Y) = \sum_{s=\tau_0}^{\tau-1} \sum_{k=0}^{\ell_0} P_s(D_s,D_s+y_k) \left([B_sy_k](\ell)\right)^2
$$
is at most
$$
R^\ell = 1600 \frac{\tau-1}{\ell^3} + (10\psi)^4 \log(\psi \tau),
$$
whenever $\tau_0 \le \tau \le \widehat{T} \land \widehat{T}_\Delta$ and $1 \le \ell \le \ell_0$.

Recall that, for
$s < \tau \le \widehat{T} \land \widehat{T}_\Delta$, and $1 \le k \le \ell_0$, we have
\begin{equation} \label{dsk}
D_s(k) \le \begin{cases} 0 & k > \psi \sqrt{s-1} \\
5 \frac{s}{k^3} + 400 \psi^2 \log(\psi s) & k \le \psi \sqrt{s-1} \end{cases}.
\end{equation}
For this section, we may and shall assume that we do indeed have these bounds on the values of $D_s(k)$.

Recall also that, for $s=\tau_0, \dots, \tau-1$, $B_s$ is the matrix product $A_{\tau-1}\cdots A_{s+1}$, and that
$$
y_k = \begin{cases} e_1 & k=0 \\
e_1 - e_k + e_{k+1} & 1 \le k < \ell_0 \\
e_1 - e_{\ell_0} & k = \ell_0.
\end{cases}
$$
We may now write, for $1 \le \ell \le \ell_0$, and $\tau_0 \le s < \tau$,
\begin{eqnarray*}
{}[B_s y_0](\ell) &=& [B_s e_1](\ell)\\
& =& B_s(\ell,1) \\
{}[B_s y_k](\ell) &=& [B_s e_1](\ell) - [B_s e_k](\ell) + [B_s e_{k+1}](\ell) \\
&=& B_s(\ell,1) - B_s(\ell,k) + B_s(\ell,k+1) \quad (1 \le k < \ell_0) \\
{}[B_s y_{\ell_0}](\ell) &=& [B_s e_1] (\ell) - [B_s e_{\ell_0}](\ell) \\
&=& B_s(\ell,1) - B_s(\ell, \ell_0).
\end{eqnarray*}
where $[B_s](i,j)$ denotes the $(i,j)$-entry of the matrix $B_s$.  We then have
\begin{eqnarray*}
{}[B_s y_0](\ell)^2 &=& B_s(\ell,1)^2 \\
{}[B_s y_k](\ell)^2 &\le& 2 B_s(\ell,1)^2 + 2(B_s(\ell,k) - B_s(\ell,k+1))^2 \quad (1 \le k < \ell_0) \\
{}[B_s y_{\ell_0}](\ell)^2 &\le& 2B_s(\ell,1)^2 + 2B_s(\ell, \ell_0)^2.
\end{eqnarray*}

Provided we interpret $B_s(\ell_0,\ell_0+1)$ as equal to zero,
we can now bound the sum over $k$, for any $s$ and any $\ell \le \ell_0$, as
\begin{eqnarray*}
\lefteqn{ \sum_{k=0}^{\ell_0} P_s (D_s,D_s+y_k) {\left([B_s y_k](\ell)\right)}^2 } \\
&\le& P_s(D_s,D_s+y_0) B_s(\ell,1)^2 \\
&&\mbox{} + 2\sum_{k=1}^{\ell_0} P_s (D_s,D_s+y_k)
\left[B_s(\ell,1)^2 + (B_s(\ell,k) - B_s(\ell,k+1))^2\right] \\
&\le& 2B_s(\ell,1)^2 + 2 \sum_{k=1}^{\ell_0} P_s (D_s,D_s+y_k)
(B_s(\ell,k) - B_s(\ell,k+1))^2.
\end{eqnarray*}

\smallskip

For $1\le \ell < \ell_0$, all terms in the sum with $k > \ell$ are zero, since the matrix $B_s$ is lower-triangular, and
therefore we have
$$
\Phi^\ell_{\tau-1}(Y) \le 2 \sum_{s=\tau_0}^{\tau-1}
B_s(\ell,1)^2 + 2 \sum_{s=\tau_0}^{\tau-1} \sum_{k=1}^{\ell} P_s (D_s,D_s+y_k) (B_s(\ell,k) - B_s(\ell,k+1))^2 .
$$
The key task is thus to estimate the entries $B_s(\ell,k)$ of the matrix product $B_s = A_{\tau-1}\cdots A_{s+1}$, and in
particular the differences $|B_s(\ell,k) - B_s(\ell,k+1)|$.
The recurrence satisfied by these matrix entries is that, for $0\le j < \ell$:
\begin{eqnarray*}
\lefteqn{B_{s-1}(\ell,\ell-j) = [B_sA_s](\ell,\ell-j)} \\
&=& B_s(\ell,\ell-j) A_s(\ell-j,\ell-j) + B_s(\ell,\ell-j+1) A_s(\ell-j+1,\ell-j),
\end{eqnarray*}
since the only non-zero entries of $A_s$ in column~$(\ell-j)$ are those in rows $(\ell-j)$ and $(\ell-j+1)$.  Substituting for the
values of these entries yields
$$
B_{s-1}(\ell,\ell-j) = B_s(\ell,\ell-j) \left( 1 - \frac{\ell-j}{2s} \right) + B_s(\ell,\ell-j+1) \frac{\ell-j}{2s}.
$$
For notational convenience, we fix $\ell \ge 1$ and write
$$
a_j(s) = a_j^{(\ell)}(s) = B_s(\ell,\ell-j),
$$
for $s =\tau_0, \dots, \tau-1$ and $j=-1, 0,\dots,\ell-1$.

Rewriting in terms of the $a_j(s)$ gives:
\begin{equation} \label{bound}
\Phi^\ell_{\tau-1}(Y) \le 2 \sum_{s=\tau_0}^{\tau-1}
a_{\ell-1}(s)^2 + 2 \sum_{s=\tau_0}^{\tau-1} \sum_{k=1}^{\ell} P_s (D_s,D_s+y_k) (a_{\ell-k}(s) - a_{\ell-k-1}(s))^2 .
\end{equation}
The transition probabilities $P_s(D_s,D_s+y_k)$ can be expressed explicitly as $\displaystyle \frac{k D_s(k)}{2(s-1)}$ for
each $s$ and $k$.

The recurrence satisfied by the $a_j(s)$ is then:
\begin{eqnarray*}
a_j(s-1) = \frac{\ell - j}{2s} a_{j-1}(s) + \left( 1 - \frac{\ell -j}{2s}\right) a_j(s),
\end{eqnarray*}
for $0 \le j \le \ell-1$, and $\tau_0 \le s \le \tau-1$.
We also have $B_{\tau-1} = I$, the identity matrix, so that $a_0(\tau-1) = 1$, and $a_j(\tau-1)=0$ for $j>0$.
Note also that $a_{-1}(s) = 0$ for all $s$, since the matrix $B_s$ is lower-triangular.  These boundary
conditions, together with the recurrence relation, suffice to determine all the values $a_j(s)$.

There is a natural interpretation of the term $a_j(s)$: it is the probability that a fixed vertex $v$ with degree $\ell - j$ at time
$s$ will have degree $\ell$ at time $\tau-1$.  This can most easily be seen by checking that this system of probabilities satisfies
the boundary conditions and the recurrence relation.  In the notation of Section~\ref{Seve},
$$
a_j(s) = \P (X_{\tau-1}(v) = \ell \mid X_s(v) = \ell-j).
$$
One immediate consequence is that $0 \le a_j(s) \le 1$ for all $j$ and $s$.

\smallskip

It may be of interest to note that there is a formula for the $a_j(s)$ as an alternating sum:
$$
a_j(s) = \binom{\ell-1}{j} \sum_{i=\ell-j}^\ell \binom{j}{\ell-i} (-1)^{i-\ell+j} \prod_{u=s+1}^{\tau-1} \left( 1 - \frac{i}{2u} \right).
$$
One may verify that this formula satisfies the recurrence.  It can also be obtained by observing that the matrices $A_s$ can be simultaneously
diagonalised, leading to a formula for the matrix $B_s$.  We also obtain
$$
a_{\ell-k}(s) - a_{\ell-k-1}(s) = \binom{\ell-1}{k} \frac{1}{\ell-k} \sum_{i=k}^\ell i \binom{\ell-k}{i-k} (-1)^{i-k}
\prod_{u=s+1}^{\tau-1} \left( 1 - \frac{i}{2u} \right).
$$
Although these formulae are quite appealing, we have been unable to extract useful bounds from them.

\smallskip

At this point, we break into three cases.  The main case of interest is when $8 \le \ell \le 2 \psi \sqrt{\tau-1}$, but we also need
to deal with values of $\ell$ outside this range, and we do this first.

For $\ell \le 7$, all we have to do is note that, from (\ref{bound}),
$$
\Phi^\ell_{\tau-1}(Y) \le 4 (\tau-1) \le 1600 \frac{\tau-1}{\ell^3}.
$$

\smallskip

Now suppose that $\ell > 2 \psi \sqrt{\tau-1}$.
By assumption, whenever $k > \psi \sqrt{s-1}$, we have $D_s(k) =0$ and so $P_s(D_s,D_s+y_k) =0$, and such terms contribute nothing
to the double sum.  We now need to bound the
contribution of terms where $k \le \psi \sqrt {s-1}$ and $\ell > 2 \psi \sqrt{\tau-1}$.  To do this, we use the inequality
$(a_{\ell-k}(s) - a_{\ell-k-1}(s))^2 \le a_{\ell-k}(s)^2 + a_{\ell-k-1}(s)^2$, and bound the size of any
term $a_{\ell-k}(s)$ subject to the given conditions.  For this, we observe that
\begin{eqnarray*}
a_{\ell-k}(s) &\le& \P (X_{\tau-1}(v) \ge \ell \mid X_s(v) = k) \\
&\le& \P \left(X_{\tau-1}(v) \ge 2\psi \sqrt{\tau-1} \mid X_s(v) = \left\lfloor \psi \sqrt{s-1} \right\rfloor \right).
\end{eqnarray*}
We now apply Corollary~\ref{cor:2times}, with $\tau_0$ replaced by $s$, $m_0$ replaced by $\lfloor \psi \sqrt {s-1}\rfloor$,
and $\omega$ replaced by $(s-1)^{1/14}$.  We have $\lfloor \psi \sqrt {s-1}\rfloor \ge 10^5 (s-1)^{7/14}$, since $\psi > 2\times 10^5$.
We also have $(s-1)^{1/14} \ge 4$ provided $s> 2^{28}$.  So, for $s > 2^{28}$, we have
$$
a_{\ell-k}(s) \le e^{-\frac54 (s-1)^{1/14}}.
$$

Thus, for each $s > 2^{28}$,
$$
\sum_{k=1}^{\lfloor\psi \sqrt {s-1}\rfloor} P_s (D_s,D_s+y_k) a_{\ell-k}(s)^2 \le e^{-\frac52 (s-1)^{1/14}}.
$$
For $s < 2^{28}$, we have
$$
\sum_{k=1}^{\lfloor\psi \sqrt {s-1}\rfloor} P_s (D_s,D_s+y_k) a_{\ell-k}(s)^2 \le 1.
$$

Therefore
\begin{eqnarray*}
\Phi^\ell_{\tau-1}(Y) &\le& 2 \sum_{s=\tau_0}^{\tau-1} \left[ a_{\ell-1}(s)^2 +
2 \sum_{k=1}^{\lfloor\psi \sqrt {s-1}\rfloor} P_s (D_s,D_s+y_k) a_{\ell-k}(s)^2 \right] \\
&\le& 6 \left[ 2^{28} + \sum_{s=2^{28}+1}^\infty e^{-\frac52 (s-1)^{1/14}} \right] \\
&\le& 2 \times 10^9,
\end{eqnarray*}
which, since $\psi \ge 2 \times 10^5$, is at most $\psi^2$.  This comfortably gives the required result
in the case where $\ell > 2 \psi \sqrt {\tau-1}$.

\smallskip

For the remainder of this section, we assume that $8 \le \ell \le 2 \psi \sqrt{\tau-1}$.

Although our exact expression for the $a_j(s)$ proved difficult to work with, we now give a function $f_j(s)$ which has a
simple form, and which satisfies the boundary conditions and an approximate version of the recurrence; our plan is to
show that $a_j(s)$ is close to $f_j(s)$ for all values of $j$ and $s$.

For $0 \le j \le \ell -1$ and $0\le s\le \tau-1$, set
$$
f_j(s) = \binom{\ell-1}{j} \left( 1 - \sqrt{\frac{s}{\tau-1}} \right)^j \sqrt{\frac{s}{\tau-1}}^{\ell -j}.
$$
Throughout what follows, we shall set $v=v_s = \sqrt{s/(\tau-1)}$, so
$$
f_j(s) = \binom{\ell-1}{j} (1-v)^j v^{\ell-j}.
$$

We note that $v_{\tau-1} =1$, and so $f_j(\tau-1)=0$ for $j\not=0$, while $f_0(\tau-1) = 1$.  We could formally define
the function $f_{-1}$ to be identically~0: the key identity we use for the binomial coefficients is
$\displaystyle \binom{\ell-1}{j-1} = \binom{\ell-1}{j} \frac{j}{\ell-j}$, which indeed entails $\binom{\ell-1}{-1} =0$.
However, we find ourselves having to deal with the case $j=0$ as a boundary case separately anyway, and so we need make no (further)
explicit mention of the case $j=-1$.

We claim that
$$
f_j(s-1) = \frac{\ell-j}{2s} f_{j-1}(s) + \left( 1 - \frac{\ell-j}{2s}\right) f_j(s) + \Big[ f_j(s-1) - f_j(s) + f'_j(s) \Big],
$$
for all $j \ge 1$ and $1 \le s \le \tau-1$.  Our aim will then be to show that the term in square brackets is usually small, and
that this is thus a good approximation to the recurrence satisfied by the $a_j(s)$.
Rearranging the claimed identity, we see that it is equivalent to
\begin{equation} \label{defj}
f'_j(s) = \frac{\ell-j}{2s} (f_j(s) - f_{j-1}(s)).
\end{equation}
To verify this identity, we write
\begin{eqnarray}
f'_j(s) &=& \binom{\ell-1}{j} v'(s) \frac{d}{dv} \left( (1-v)^j v^{\ell-j} \right) \big\vert_{v=v_s} \nonumber\\
&=& \binom{\ell-1}{j} \frac{v}{2s} (1-v)^{j-1} v^{\ell-j-1} \left((\ell-j)(1-v) - jv\right) \label{fprime} \\
&=& \frac{\ell-j}{2s} (1-v)^{j-1} v^{\ell-j} \left( \binom{\ell-1}{j} (1-v) - \frac{j}{\ell-j} \binom{\ell-1}{j} v \right) \nonumber\\
&=& \frac{\ell-j}{2s} \left( \binom{\ell-1}{j} (1-v)^j v^{\ell-j} - \binom{\ell-1}{j-1} (1-v)^{j-1} v^{\ell-j+1} \right)\nonumber\\
&=& \frac{\ell-j}{2s} (f_j(s) - f_{j-1}(s)).\nonumber
\end{eqnarray}

Equation (\ref{defj}) demonstrates that the $f_j(s)$ are the analogues to the $a_j(s)$ for a continuous time version of the preferential
attachment process.  In this continuous time version, at time $s$, each vertex of degree $k$ attracts a new edge (whose other endpoint
is a new vertex of degree~1) at rate $k/2s$, independent of the degrees of other vertices.  The degree of a given vertex is then a
pure birth process with this transition rate.  The probability that a vertex with degree $\ell-j$ at time $s$ has degree $\ell$ at time
$\tau-1$ satisfies the differential equation (\ref{defj}), as well as the boundary condition $f_j(\tau-1) = \delta_{j0}$.

It seems intuitively plausible that the difference $e_j(s) = f_j(s) - a_j(s)$ between the continuous and the discrete ``solutions''
will always be small.  Indeed we shall prove the following lemma, which is very crude in most ranges.

\begin{lemma} \label{lem:ejs}
For all $\ell \ge 8$ and $0 \le j \le \ell -1$, we have:
$$
|e_j(s)| \le \begin{cases} \frac{2200\ell}{\tau-1} & (\tau-1)/\ell^2 < s \le \tau-1 \mbox{ or } j \le \ell-2 \\
\frac{800\ell^{3/2}}{\tau-1} & (\tau-1) /\ell^3 <  s \le (\tau-1) /\ell^2 \\
1 & \tau_0 \le s \le (\tau-1)/\ell^3. \end{cases}
$$
\end{lemma}

We shall defer the proof of Lemma~\ref{lem:ejs} to the next section.

\smallskip

We set
$$
\Psi^\ell_{\tau-1}(Y) = 4 \sum_{s=\tau_0}^{\tau-1}
f_{\ell-1}(s)^2 + 4 \sum_{s=\tau_0}^{\tau-1} \sum_{k=1}^{\ell} P_s (D_s,D_s+y_k) (f_{\ell-k}(s) - f_{\ell-k-1}(s))^2 .
$$
We now show that the bound in Lemma~\ref{lem:ejs} suffices to show that $\Phi^\ell_{\tau-1}(Y)$ is not much
larger than $\Psi^\ell_{\tau-1}(Y)$.

\begin{lemma} \label{PhivPsi}
For any $\ell$ and $\tau$, with $8 \le \ell \le 2\psi \sqrt{\tau-1}$,
$$
\Phi^\ell_{\tau-1}(Y) \le \Psi^\ell_{\tau-1}(Y) + 5 \times 10^8 \psi^2 + 20 \frac{\tau-1}{\ell^3}.
$$
\end{lemma}

\begin{proof}
Equation (\ref{bound}) tells us that $\Phi^\ell_{\tau-1}(Y)$ is at most
$$
2 \sum_{s=\tau_0}^{\tau-1} a_{\ell-1}(s)^2 +
2 \sum_{s=\tau_0}^{\tau-1} \sum_{k=1}^{\ell} P_s (D_s,D_s+y_k) (a_{\ell-k}(s) - a_{\ell-k-1}(s))^2 .
$$
Using the inequalities $a_j(s)^2 \le 2f_j(s)^2 + 2e_j(s)^2$ and
$$
(a_j(s)-a_{j-1}(s))^2 \le 2(f_j(s) - f_{j-1}(s))^2 + 4 e_j(s)^2 + 4 e_{j-1}(s)^2,
$$
we deduce that
\begin{eqnarray*}
\Phi^\ell_{\tau-1}(Y) &\le& \Psi^\ell_{\tau-1}(Y) + 4 \sum_{s=\tau_0}^{\tau-1} e_{\ell-1}(s)^2 \\
&& \mbox{} + 8 \sum_{s=\tau_0}^{\tau-1} \sum_{k=1}^\ell P_s(D_s,D_s+y_k) \left( e_{\ell-k}(s)^2 + e_{\ell-k-1}(s)^2 \right).
\end{eqnarray*}

Now we apply the bounds from Lemma~\ref{lem:ejs}:
\begin{eqnarray*}
\sum_{s=\tau_0}^{\tau-1} e_{\ell-1}(s)^2 &\le& (\tau-1) \left( \frac{2200 \ell}{\tau-1} \right)^2 +
\frac{\tau-1}{\ell^2} \left( \frac{800\ell^{3/2}}{\tau-1} \right)^2 + \frac{\tau-1}{\ell^3} \\
&\le& \frac{5\times 10^6 \ell^2}{\tau-1} + \frac{10^6 \ell}{\tau-1} + \frac{\tau-1}{\ell^3} \\
&\le& \frac{6\times 10^6 \ell^2}{\tau-1} + \frac{\tau-1}{\ell^3},
\end{eqnarray*}
and similarly
\begin{eqnarray*}
\lefteqn{\sum_{s=\tau_0}^{\tau-1} \sum_{k=1}^{\ell} P_s (D_s,D_s+y_k) \left( e_{\ell-k}(s)^2 + e_{\ell-k-1}(s)^2 \right)}\\
&\le& 2 \left( (\tau-1) \left( \frac{2200 \ell}{\tau-1} \right)^2 +
\frac{\tau-1}{\ell^2} \left( \frac{800\ell^{3/2}}{\tau-1} \right)^2 + \frac{\tau-1}{\ell^3} \right)\\
&\le& 2 \left( \frac{6\times 10^6 \ell^2}{\tau-1} + \frac{\tau-1}{\ell^3} \right).
\end{eqnarray*}
Therefore
\begin{eqnarray*}
\Phi^\ell_{\tau-1}(Y) &\le& \Psi^\ell_{\tau-1}(Y) + 20 \left( \frac{6\times 10^6 \ell^2}{\tau-1} + \frac{\tau-1}{\ell^3} \right) \\
&\le & \Psi^\ell_{\tau-1}(Y) + \frac{120 \times 10^6 \ell^2}{\tau-1} + 20 \frac{\tau-1}{\ell^3} \\
&\le & \Psi^\ell_{\tau-1}(Y) + 5 \times 10^8 \psi^2 + 20 \frac{\tau-1}{\ell^3},
\end{eqnarray*}
as claimed.
\end{proof}

For $k = 1, \dots, \ell$, we have that $P_s(D_s,D_s+y_k) = k D_s(k) / 2(s-1)$, since each of the $D_s(k)$ vertices of
degree~$k$ has probability $k/2(s-1)$ of receiving an extra edge at time $s+1$.
Therefore
$$
\Psi^\ell_{\tau-1}(Y) = 4 \sum_{s=\tau_0}^{\tau-1} f_{\ell-1}(s)^2 +
2 \sum_{s=\tau_0}^{\tau-1} \sum_{k=1}^{\ell} \frac{k D_s(k)}{s-1} (f_{\ell-k}(s) - f_{\ell-k-1}(s))^2 .
$$
The double sum is the main term here, and we mainly concentrate on this; we will obtain adequate bounds
on $\sum_s f_{\ell-1}(s)^2$ as a byproduct of our estimates.

\smallskip

Recall our assumptions (\ref{dsk}) that $D_s(k) = 0$ for all $k > \psi \sqrt {s-1}$, and that
$\displaystyle D_s(k) \le 5\frac{s}{k^3} + 400 \psi^2 \log(\psi s)$ for all $k=1, \dots, \ell_0$ with $k \le \psi \sqrt{s-1}$.
Using these bounds, we find that, for all $k =1, \dots, \ell_0$,
and all $s \ge 4$,
$$
\frac{k D_s(k)}{s-1} \le \frac{7}{k^2} + 550 \psi^2 \log (\psi s) \frac{k}{s}.
$$
Thus we have
\begin{eqnarray*}
\lefteqn{\Psi^\ell_{\tau-1}(Y) \le 4 \sum_{s=\tau_0}^{\tau-1} f_{\ell-1}(s)^2 }\\
&&\mbox{} + 14 \sum_{s=\tau_0}^{\tau-1} \sum_{k=1}^{\ell} \frac{1}{k^2} (f_{\ell-k}(s) - f_{\ell-k-1}(s))^2 \\
&&\mbox{} + 1100 \psi^2 \log(\psi\tau)
\sum_{s=\tau_0}^{\tau-1} \sum_{k=1}^\ell \frac{k}{s} (f_{\ell-k}(s) - f_{\ell-k-1}(s))^2.
\end{eqnarray*}
We define
\begin{eqnarray*}
Q_1(\tau,\ell) &=& \sum_{s=\tau_0}^{\tau-1} \sum_{k=1}^\ell \frac{1}{k^2} \left( f_{\ell-k}(s) - f_{\ell-k-1}(s) \right)^2 \\
Q_2(\tau,\ell) &=& \sum_{s=\tau_0}^{\tau-1} \sum_{k=1}^\ell
\frac{k}{s} \left( f_{\ell-k}(s) - f_{\ell-k-1}(s) \right)^2
\end{eqnarray*}
so that
\begin{equation} \label{PsivQi}
\Psi^\ell_{\tau-1}(Y)
\le 4 \sum_{s=\tau_0}^{\tau-1} f_{\ell-1}(s)^2 + 14 Q_1(\tau,\ell) + 1100 \psi^2 \log(\psi\tau) Q_2(\tau,\ell).
\end{equation}

\smallskip

To estimate $Q_1$, we exchange the order of summation and substitute $j=\ell-k$:
$$
Q_1(\tau,\ell) = \sum_{j=0}^{\ell-1} \frac{1}{(\ell-j)^2} \sum_{s=\tau_0}^{\tau-1} \left( f_j(s) - f_{j-1}(s) \right)^2.
$$
From (\ref{defj}) and (\ref{fprime}), we have
$$
f_j(s) - f_{j-1}(s) = \binom{\ell-1}{j} \frac{1}{\ell-j} (1-v)^{j-1} v^{\ell-j} \left( \ell(1-v) - j \right),
$$
where $v = v_s = \sqrt{s/(\tau-1)}$, as before.  We estimate the sum over $s$ by approximating it by the integral
$$
\int_{s=\tau_0}^{\tau-1} \binom{\ell-1}{j}^2 \frac{1}{(\ell-j)^2}(1-v)^{2j-2} v^{2\ell-2j} \left(\ell(1-v)-j\right)^2 \, ds.
$$
The integrand here is bounded above by~1, since each $f_j(s)$ is at most~1.  The function
$(1-v)^{j-1} v^{\ell-j} \left( \ell(1-v) - j \right)$ has derivative which is a positive multiple of a quadratic function of
$v$, so the function has just two stationary points, one either side of the zero $v = (\ell-j)/\ell$.  Therefore the integrand,
which is a positive multiple of the square of this function, has two local maxima.  The sum is then at most the value of the
integral plus the values of the integrand at the two local maxima, and so
\begin{eqnarray*}
\lefteqn{Q_1(\tau,\ell) \le 2 \sum_{j=0}^{\ell-1} \frac{1}{(\ell-j)^2}} \\
&& \mbox{} + \sum_{j=0}^{\ell-1} \frac{1}{(\ell-j)^4} \binom{\ell-1}{j}^2
\int_{s=\tau_0}^{\tau-1} (1-v)^{2j-2} v^{2\ell-2j} \left( \ell(1-v)-j \right)^2 \, ds \\
&\le & 4 + \sum_{j=0}^{\ell-1} \frac{2(\tau-1)}{(\ell-j)^4} \binom{\ell-1}{j}^2
\int_{v=0}^1 (1-v)^{2j-2} v^{2\ell-2j+1} (\ell (1-v) - j)^2 \, dv.
\end{eqnarray*}
In the last line, we changed variable: recall that $s=v^2(\tau-1)$.

We write
$$
Q_1(\tau,\ell) \le 4 + \sum_{j=0}^{\ell-1} \frac{2(\tau-1)}{(\ell-j)^4} \binom{\ell-1}{j}^2 I(\ell,j,1),
$$
where
$$
I(\ell,j,\alpha) = \int_{v=0}^1 (1-v)^{2j-2} v^{2\ell-2j + \alpha} \left( \ell(1-v) - j \right)^2 \, dv,
$$
for positive integers $\ell$ and $j$, and integer $\alpha$, where $\ell > j$ and $\alpha \ge -1$.

The integral above can be evaluated as a sum of Beta functions.  We will be confronted by a very similar integral
when estimating $Q_2$, and it is convenient to prove a lemma covering both cases (here we need $\alpha =1$ and later
we shall take $\alpha=-1$).

\begin{lemma} \label{lem:beta}
For integers $\ell $ and $j$ with $\ell > j \ge 0$, and integer $\alpha\ge-1$,
$$
I(\ell,j,\alpha) \le \frac{(2\ell-2j+\alpha)! (2j-2)!}{(2\ell+\alpha+1)!} j\ell \left\{ 2(\ell-j+1) + \alpha (3+\alpha)\right\}
\quad (j \ge 1)
$$
and
$$
I(\ell,0,\alpha) \le \frac{\ell}{2}.
$$
\end{lemma}

\begin{proof}
For non-negative integers $a$ and $b$, we have the identity
$$
\int_{v=0}^1 (1-v)^a v^b \, dv = B(a+1,b+1) = \frac{a! \, b!}{(a+b+1)!},
$$
where $B(\cdot,\cdot)$ denotes the Beta function.

For $j\ge 1$, the required integral can be written as a sum of three integrals of the form above, and we obtain
\begin{eqnarray*}
\lefteqn{I(\ell,j,\alpha)}\\
&=& (2\ell -2j+\alpha)! \left( \ell^2 \frac{(2j)!}{(2\ell+\alpha+1)!} - 2 \ell j \frac{(2j-1)!}{(2\ell+\alpha)!}
+ j^2 \frac{(2j-2)!}{(2\ell+\alpha -1)!} \right) \\
&=& (2\ell-2j+\alpha)! \frac{(2j-2)!}{(2\ell+\alpha +1)!} \\
&& \mbox{} \times \left\{ \ell^2(2j-1)(2j) - 2j \ell(2j-1)(2\ell+\alpha+1) + j^2 (2\ell+\alpha+1)(2\ell+\alpha) \right\} \\
&=& \frac{(2\ell-2j+\alpha)! (2j-2)!}{(2\ell+\alpha+1)!} j \left\{ 2\ell(\ell-j+1) + \alpha(j(1+\alpha) + 2\ell) \right\}\\
&\le& \frac{(2\ell-2j+\alpha)! (2j-2)!}{(2\ell+\alpha+1)!} j\ell \left\{ 2(\ell-j+1) + \alpha (3+\alpha)\right\},
\end{eqnarray*}
as claimed.

For $j=0$, we have
$$
I(\ell, 0, \alpha) = \ell^2 \int_{v=0}^1 v^{2\ell +\alpha} \, dv = \frac{\ell^2}{2\ell + \alpha + 1} \le \frac{\ell}{2},
$$
for all $\alpha \ge -1$, also as claimed.
\end{proof}

Lemma~\ref{lem:beta}, with $\alpha = 1$, tells us that
\begin{eqnarray*}
\lefteqn{Q_1(\tau,\ell) \le 4 + 2\frac{\tau-1}{\ell^4} I(\ell,0,1)} \\
&& \mbox{} + 2(\tau-1) \sum_{j=1}^{\ell-1} \frac{j\ell\{2(\ell-j+1)+4\}} {(\ell-j)^4}
\binom{\ell-1}{j}^2 \frac{(2\ell-2j+1)!\, (2j-2)!}{(2\ell+2)!}\\
&=&
4 + \frac{\tau-1}{\ell^3} + 4(\tau-1) \sum_{j=1}^{\ell-1} \frac{j\ell(\ell-j+3)} {(\ell-j)^4}
\frac{(\ell-1)!^2}{(2\ell+2)!}
\frac{(2\ell-2j+1)!}{(\ell-j-1)!^2} \frac{(2j-2)!}{j!^2}\\
&=& 4 + \frac{\tau-1}{\ell^3} + 4(\tau-1) \sum_{j=1}^{\ell-1} \binom{2\ell}{\ell}^{-1} \frac{1}{\ell (2\ell+2)(2\ell+1)} \\
&&\mbox{} \qquad \qquad \qquad \qquad \binom{2\ell-2j}{\ell-j}
\frac{(\ell-j+3) (2\ell-2j+1)}{(\ell-j)^2} \binom{2j}{j} \frac{j}{2j(2j-1)}.
\end{eqnarray*}

This is the first of several occasions in the paper where we use the inequalities
$$
\frac{2^{2x}}{2\sqrt x} \le \binom{2x}{x} \le \frac{2^{2x}}{\sqrt{x+1}};
$$
the first is valid for all integers $x \ge 1$, and the second for all integers $x \ge 0$.  Sometimes, as below,
we use simply that $\displaystyle \binom{2x}{x} \le \frac{2^{2x}}{\sqrt x}$.

We obtain
\begin{eqnarray*}
Q_1(\tau,\ell) &\le& 4 + \frac{\tau-1}{\ell^3} + 4(\tau-1) \sum_{j=1}^{\ell-1} \frac{2\sqrt \ell}{2^{2\ell}} \frac{1}{4\ell^3}
\frac{2^{2\ell -2j}} {\sqrt{\ell-j}} 12 \frac{2^{2j}}{\sqrt j} \frac{1}{2j} \\
&=& 4 + \frac{\tau-1}{\ell^3} + \frac{12(\tau-1)}{\ell^3} \sum_{j=1}^{\ell-1} \frac{1}{j^{3/2}}
\sqrt{\frac{\ell}{\ell-j}}.
\end{eqnarray*}

To estimate the sum appearing above, we use the numerical value $\sum_{j=1}^\infty j^{-3/2} \le 2.61238$, and the crude bound
$$
\sum_{j=1}^{\ell-1} \frac{1}{j^{3/2}} \sqrt{\frac{\ell}{\ell-j}} \le 2 \sum_{j=1}^{\ell/2} \frac{1}{j^{3/2}} \sqrt{\frac{\ell}{\ell-j}}
\le 2 \sqrt 2 \sum_{j=1}^\infty \frac{1}{j^{3/2}} \le 8,
$$
and obtain
\begin{equation} \label{Q1}
Q_1(\tau,\ell) \le 4 + \frac{100(\tau-1)}{\ell^3}.
\end{equation}

\smallskip

The next step is to estimate
\begin{eqnarray*}
Q_2(\tau,\ell) &=& \sum_{s=\tau_0}^{\tau-1} \sum_{k=1}^\ell \frac{k}{s} \left( f_{\ell-k}(s) - f_{\ell-k-1}(s) \right)^2 \\
&\le& \sum_{j=0}^{\ell-1} (\ell-j)
\sum_{s=\tau_0}^{\tau-1} \frac{1}{s} \left( f_j(s) - f_{j-1}(s) \right)^2.
\end{eqnarray*}
As before, we shall start by fixing $j$, and estimating the sum over $s$ by the integral
\begin{eqnarray*}
&& \int_{s=\tau_0}^{\tau-1} \frac{1}{s} \left( f_j(s) - f_{j-1}(s) \right)^2 \, ds \\
&& \mbox{} \le \binom{\ell-1}{j}^2 \frac{1}{(\ell-j)^2} \int_{v=0}^1 \frac{(1-v)^{2j-2} v^{2\ell-2j}(\ell(1-v)-j)^2}{v^2(\tau-1)}
2(\tau-1)v \, dv \\
&& \mbox{} = 2 \binom{\ell-1}{j}^2 \frac{1}{(\ell-j)^2} I(\ell,j,-1).
\end{eqnarray*}
We used the expression for $f_j(s)-f_{j-1}(s)$ derived earlier, and made the substitution
$s=v^2(\tau-1)$.

\smallskip

To bound the difference between the sum $\sum_{s=\tau_0}^{\tau-1} \frac{1}{s} \left( f_j(s) - f_{j-1}(s) \right)^2$ and
the corresponding integral is not completely straightforward.  The integrand $\frac{1}{s} \left( f_j(s) - f_{j-1}(s) \right)^2$
can be written as
$$
H_j(v) = \frac{h_j(v)^2}{\tau-1}, \mbox{ where }
h_j(v) = \frac{1}{\ell-j} \binom{\ell -1}{j} (1-v)^{j-1} v^{\ell-j-1} (\ell(1-v)-j).
$$

The function $h_j(v)$ has stationary points at
$$
v_* = \frac{\ell-j}{\ell} \pm \frac{1}{\ell} \sqrt{ \frac{j(\ell-j)}{\ell-1}}.
$$
Therefore $H_j(v)$ has a global minimum at $v=(\ell-j)/\ell$, and local maxima at the two points $v_*$, and so the
global maximum of $H_j(v)$ is attained at one of the $v_*$.  For $j\ge 1$, we can write
\begin{eqnarray*}
|h_j(v_*)| &=& \frac{\ell-1}{j(\ell-j)} \left[ \binom{\ell-2}{j-1} (1-v_*)^{j-1} v_*^{\ell-j-1} \right] |\ell(1-v_*)-j| \\
&\le& \frac{\ell-1}{j(\ell-j)} \sqrt{\frac{j(\ell-j)}{\ell-1}}\\
&=& \sqrt{\frac{\ell-1}{j(\ell-j)}},
\end{eqnarray*}
where we used the fact that the term in square brackets is the probability that a Binomial random variable with parameters
$(\ell-2,v_*)$ takes the value $\ell-j-1$, and is therefore at most~1.
Hence we have
$$
H_j(v) \le \frac{\ell-1}{j(\ell-j)} \frac{1}{\tau-1},
$$
for all $j\ge 1$ and all $v$.
For $j=0$, the maximum value of $h_j(v)$ is $1$, and thus $H_0(v)$ is at most $\frac{1}{\tau-1}$ for all $v$.

We have that
\begin{eqnarray*}
Q_2(\tau,\ell) &\le& \sum_{j=0}^{\ell-1} (\ell-j) \left[ 2 \binom{\ell-1}{j}^2 \frac{1}{(\ell-j)^2} I(\ell,j,-1) + 2 \max_v H_j(v) \right] \\
&\le& 2 \sum_{j=0}^{\ell-1} \binom{\ell-1}{j}^2 \frac{1}{\ell-j} I(\ell,j,-1) +
\frac{2}{\tau-1} \left\{ \ell + \sum_{j=1}^{\ell-1} \frac{\ell-1}{j} \right\} \\
&\le& 2 \sum_{j=0}^{\ell-1} \binom{\ell-1}{j}^2 \frac{1}{\ell-j} I(\ell,j,-1) + \frac{2\ell^2}{\tau-1}.
\end{eqnarray*}
Now we use the bounds for $I(\ell,j,-1)$ from Lemma~\ref{lem:beta}.  We also use that
$\ell \le 2\psi \sqrt {\tau-1}$, and $\psi \ge 3$,  to obtain:
\begin{eqnarray}
Q_2(\tau,\ell) &\le& \frac{2\ell^2}{\tau-1} + \frac{1}{\ell} \frac{\ell}{2} \nonumber \\
&&\mbox{} + 2 \sum_{j=1}^{\ell-1} \binom{\ell-1}{j}^2 \frac{1}{\ell-j} \frac{(2\ell-2j-1)!\,(2j-2)!}{(2\ell)!} 2 j\ell (\ell-j) \nonumber \\
&=& 8\psi^2 + \frac{1}{2} + 4 \sum_{j=1}^{\ell-1} \frac{(\ell-1)!^2}{j!^2 (\ell-j-1)!^2}
\frac{(2\ell-2j-1)!\,(2j-2)!}{(2\ell)!} j\ell \nonumber \\
&=& 8\psi^2 + \frac{1}{2} + 4 \sum_{j=1}^{\ell-1} \binom{2\ell}{\ell}^{-1} \frac{1}{\ell} \binom{2\ell-2j-2}{\ell-j-1} (2\ell-2j-1)
\binom{2j-2}{j-1} \frac{1}{j} \nonumber \\
&\le& 8\psi^2 + \frac{1}{2}+ 4 \sum_{j=1}^{\ell-1} \left(\frac{2\ell-2j-1}{j\ell}\right)
\left(\frac{2\sqrt \ell}{2^{2\ell}}\right) \left(\frac{2^{2\ell -2j-2}}{\sqrt{\ell-j}}\right) \left(\frac{2^{2j-2}}{\sqrt j}\right) \nonumber \\
&\le& 8\psi^2 + \frac{1}{2}+ \sum_{j=1}^{\ell-1} \frac{1}{j^{3/2}} \sqrt{\frac{\ell-j}{\ell}} \nonumber\\
&\le& 8\psi^2 + 5 \le 9\psi^2. \label{Q2}
\end{eqnarray}

\smallskip

The next task is to bound the sum
$$
\sum_{s=\tau_0}^{\tau-1} f_{\ell-1}(s)^2 =
\sum_{s=\tau_0}^{\tau-1} (1-v)^{2\ell-2} v^2,
$$
where, as before, $v =\sqrt{s/\tau-1}$.  This sum is bounded above by the integral $\int_{s=0}^{\tau-1} (1-v)^{2\ell-2} v^2 \, dv$,
plus the maximum value of the integrand.  The integral is equal to
$$
2(\tau-1) \int_{v=0}^1 (1-v)^{2\ell-2} v^3 \, dv = 2(\tau-1) \frac{(2\ell-2)! 3!}{(2\ell+2)!} \le \frac{12 (\tau-1)}{(2\ell)^4},
$$
which is more than small enough for our purposes, and the integrand is certainly at most~1, so
\begin{equation} \label{Q3}
\sum_{s=\tau_0}^{\tau-1} f_{\ell-1}(s)^2 \le 1 + \frac{\tau-1}{\ell^4}.
\end{equation}

\smallskip

Finally, we combine all our estimates.  For any $\ell =1, \dots, \ell_0$ we have, by Lemma~\ref{PhivPsi},
(\ref{PsivQi}), (\ref{Q1}), (\ref{Q2}) and (\ref{Q3}), that
\begin{eqnarray*}
\Phi^\ell_{\tau-1}(Y) &\le& \Psi^\ell_{\tau-1}(Y) + 5 \times 10^8 \psi^2 +20 \frac{\tau-1}{\ell^3}\\
&\le & 4 \sum_{s=\tau_0}^{\tau-1} f_{\ell-1}(s)^2 + 14 Q_1(\tau,\ell) +
1100 \psi^2 \log(\psi\tau) Q_2(\tau,\ell) \\
&& \mbox{} + 5 \times 10^8 \psi^2 +20 \frac{\tau-1}{\ell^3}\\
&\le& 4 + \frac{4(\tau-1)}{\ell^4} + 56 + 1400 \frac{\tau-1}{\ell^3} + 9900 \psi^4 \log(\psi\tau) \\
&&\mbox{} + 5 \times 10^8 \psi^2 +20 \frac{\tau-1}{\ell^3}\\
&<& 1600 \frac{\tau-1}{\ell^3} + 10^4\psi^4 \log(\psi\tau).
\end{eqnarray*}

Thus $\Phi^\ell_{\tau-1}(Y) < R^\ell$, as required.  This completes the proof of Lemma~\ref{PhivR}, except for the
proof of Lemma~\ref{lem:ejs}, to which the next section is devoted.

\section{Proof of Lemma~\ref{lem:ejs}} \label{Slemma}

Our aim in this section is to prove the following upper bound on $e_j(s) = f_j(s) - a_j(s)$, to be valid whenever
$\ell \ge 8$, $0\le j \le \ell-1$ and $\tau_0 \le s \le \tau-1$:
$$
|e_j(s)| \le \begin{cases} \frac{2200\ell}{\tau-1} & (\tau-1)/\ell^2 < s \le \tau-1 \mbox{ or } j \le \ell-2 \\
\frac{800\ell^{3/2}}{\tau-1} & (\tau-1) /\ell^3 <  s \le (\tau-1) /\ell^2 \\
1 & \tau_0 \le s \le (\tau-1)/\ell^3. \end{cases}
$$
The final case is straightforward, since both $a_j(s)$ and $f_j(s)$ lie between~0 and~1 for all $0\le j \le \ell-1$ and $\tau_0 \le s \le \tau-1$.
So from now on we assume that, if $j = \ell-1$, then $s > (\tau-1)/\ell^3$.

\smallskip

Recall that $a_j(s)$ and $f_j(s)$ satisfy the recurrences:
\begin{eqnarray*}
a_j(s-1) &=& \frac{\ell - j}{2s} a_{j-1}(s) + \left( 1 - \frac{\ell -j}{2s}\right) a_j(s), \\
f_j(s-1) &=& \frac{\ell-j}{2s} f_{j-1}(s) + \left( 1 - \frac{\ell-j}{2s}\right) f_j(s) + \Big[ f_j(s-1) - f_j(s) + f'_j(s) \Big],
\end{eqnarray*}
for all $j \ge 1$ and all $s$ with $\tau_0 < s \le \tau$.

The term in square brackets is, by Taylor's Theorem, equal to $\frac12 f''_j(w)$ for some $w \in (s-1,s)$.  We will
thus estimate it by bounding the absolute value of the second derivative of $f_j$.

\begin{lemma} \label{lem:f''}
For all $1\le j\le \ell-2$ and $\tau_0 \le s\le \tau-1$,
\begin{equation} \label{eq:second}
|f''_j(s)| \le \frac{140}{s(\tau-1)} \frac{(\ell-1)^{5/2}}{j^{3/2}(\ell-j)^{1/2}}.
\end{equation}
This bound also holds if $j=\ell-1$ and $s > (\tau-1)/\ell^2$.

For $(\tau-1)/\ell^3 \le s \le (\tau-1)/\ell^2$, we have
$$
|f''_{\ell-1}(s)| \le \ell^{3/2}/s(\tau-1).
$$
\end{lemma}

\begin{proof}
Using the expression for $f'_j(s)$ in (\ref{fprime}), as well as the
identities $s = v^2(\tau-1)$ and $(\ell-j)(1-v) -jv = \ell(1-v) -j$, we can write
$$
f'_j(s) = \binom{\ell-1}{j} \frac{1}{2(\tau-1)} (1-v)^{j-1} v^{\ell-j-2} (\ell(1-v) - j),
$$
and then we have:
\begin{eqnarray*}
f''_j(s) &=& \binom{\ell-1}{j} \frac{1}{2(\tau-1)} \frac{v}{2s} (1-v)^{j-2} v^{\ell-j-3} \\
&& \mbox{} \times \big[ (\ell(1-v)-j)\left\{ (\ell-j-2)(1-v) -(j-1)v \right\} -\ell v(1-v) \big] \\
&=& \binom{\ell-1}{j} \frac{1}{4s(\tau-1)} (1-v)^{j-2} v^{\ell-j-2} \\
&& \mbox{} \times \left[ \left\{ (\ell-1)(1-v) - j\right\}^2 - (1-v)^2 - vj \right].
\end{eqnarray*}

\smallskip

Let us first verify the result for $j=1$, when we can write
$$
4s(\tau-1) f''_1(s) = (\ell -1) v^{\ell-3} [(\ell-1)(\ell-3) - v\ell(\ell-2)].
$$
The right-hand side is increasing from $v=0$ to $v = (\ell-1)(\ell-3)^2/\ell(\ell-2)^2$, and decreasing thereafter.
It is thus always at least its value at $v=1$, which is $-(\ell-1)(2\ell-3)$, and at most its value at the stationary
point, which is at most
$$
(\ell-1) 1^{\ell-3} (\ell-1)(\ell-3) \left(1 - \frac{\ell-3}{\ell-2}\right)
= \frac{(\ell-1)^2(\ell-3)}{(\ell-2)} \le (\ell-1)^2,
$$
and thus
$$
|f''_1(s)| \le \frac{1}{2s(\tau-1)} (\ell-1)^2,
$$
which is as required.

We now embark on the calculation for $2\le j \le \ell-2$.  We define a parameter $\phi = \phi(v)$ by
$v = (\ell -j - 2 -\phi)/(\ell-4)$, so $1-v = (j-2+\phi)/(\ell-4)$, and $-(j-2) \le \phi \le \ell-j-2$.  The point is that
the ``main term'' $(1-v)^{j-2}v^{\ell-j-2}$ in our expression for the second derivative of $f$ is maximised at $\phi=0$, whereas the
other term $\left[ \left\{ (\ell-1)(1-v) - j\right\}^2 - (1-v)^2 - vj \right]$ is small for small $\phi$.
We write
$$
4s(\tau-1) f''_j(s) = k_1 k_2 k_3,
$$
where
\begin{eqnarray*}
k_1 &=& \binom{\ell-1}{j} \left( \frac{j-2}{\ell-4} \right)^{j-2} \left( \frac{\ell-j-2}{\ell-4} \right)^{\ell-j-2}, \\
k_2 &=& \left( 1 + \frac{\phi}{j-2}\right)^{j-2} \left( 1 - \frac{\phi}{\ell-j-2}\right)^{\ell-j-2}, \\
k_3 &=& \left( (\ell-1) \frac{j-2+\phi}{\ell-4} -j \right)^2 - \left(\frac{j-2+\phi}{\ell-4}\right)^2 - j \frac{\ell-j-2-\phi}{\ell-4}.
\end{eqnarray*}
If $j=2$ or $j=\ell-2$, the terms in $k_1$ and $k_2$ with a power of $j-2$ or $\ell-j-2$ respectively are treated as equal to 1,
and therefore absent from the products.
We shall estimate $k_1$, $k_2$ and $k_3$ separately to start with, and then consider $k_2k_3$.

We have
\begin{eqnarray*}
k_1 &=& \binom{\ell-1}{j}\left( \frac{j-2}{\ell-4} \right)^{j-2} \left( \frac{\ell-j-2}{\ell-4} \right)^{\ell-j-2} \\
&=& \frac{(\ell-1)(\ell-2)(\ell-3)}{j(j-1)(\ell-j-1)} \left(\frac{(\ell-4)!\, e^{\ell-4}}{(\ell-4)^{\ell-4}}\right)
\left(\frac{ (j-2)^{j-2}}{(j-2)!\, e^{j-2}}\right) \\
&& \mbox{} \times \left( \frac{ (\ell-j-2)^{\ell-j-2}}{(\ell-j-2)!\, e^{\ell-j-2}} \right).
\end{eqnarray*}
Again, if $j-2$ or $\ell-j-2$ is zero, the related term is absent (i.e., the ratio is equal to~1).
We now use the inequalities
$$
\sqrt{x+1} \left( \frac{x}{e} \right)^x \le x! \le 3 \sqrt{x} \left( \frac{x}{e} \right)^x,
$$
to obtain that
$$
k_1 \le 3 \frac{(\ell-1)(\ell-2)(\ell-3)}{j(j-1)(\ell-j-1)} \sqrt{\frac{\ell-4}{(j-1)(\ell-j-1)}}.
$$
(Note that this remains valid if $j=2$ or $j=\ell-2$.)

\smallskip

We next consider $k_2$.  We assume for the moment that $j \le \ell/2$ (the other case is symmetric) and distinguish two ranges.
First, we consider the case where $|\phi| < j-2$.  In this case, we use the bound
$\log (1+x) \le x - \frac{x^2}{4}$, valid for all $|x| < 1$, and obtain:
\begin{eqnarray*}
\log k_2 &=& (j-2) \log \left( 1 + \frac{\phi}{j-2}\right) + (\ell-j-2) \log \left( 1 - \frac{\phi}{\ell-j-2}\right) \\
&\le& (j-2) \left( \frac{\phi}{j-2} - \frac{\phi^2}{4(j-2)^2} \right) \\
&& \mbox{} +
(\ell-j-2) \left( \frac{-\phi}{\ell-j-2} - \frac{\phi^2}{4(\ell-j-2)^2} \right)\\
&=& - \frac{\phi^2}{4} \left( \frac{1}{j-2} + \frac{1}{\ell-j-2} \right)\\
&=& - \frac{\phi^2}{4} \frac{\ell-4}{(j-2)(\ell-j-2)}.
\end{eqnarray*}
In the case where $\phi = \alpha (j-2)$ with $\alpha \ge 1$ and $j > 2$, we estimate
\begin{eqnarray*}
\left( 1 + \frac{\phi}{j-2}\right)^{j-2} \left( 1 - \frac{\phi}{\ell-j-2}\right)^{\ell-j-2}
&\le& (1+\alpha)^{\phi/\alpha} e^{-\phi} \\
&=& \left( (1+\alpha)^{1/\alpha} e^{-1} \right)^\phi \\
&\le& (2/e)^\phi.
\end{eqnarray*}
If $j=2$, then $k_2$ is just $(1-\phi/(\ell-4))^{\ell-4}$, which is at most $e^{-\phi} \le (2/e)^\phi$.
There is a final case where $\phi = -(j-2)$, i.e., $v=1$, and we may dispose of this immediately since the second
derivative is zero unless $j=2$.
In summary, $\displaystyle k_2 \le \exp\left( - \frac{\phi^2}{4} \frac{\ell-4}{(j-2)(\ell-j-2)} \right)$
if $|\phi| < \min (j-2, \ell-j-2)$, and $k_2 \le (2/e)^{|\phi|}$ otherwise.

\smallskip

Let us organise $k_3$ as a quadratic in $\phi$:
\begin{eqnarray*}
k_3 &=& \frac{1}{(\ell-4)^2} \Big[ \phi^2 \ell(\ell-2) - \phi (4\ell^2 -7\ell j - 8\ell + 12 j) \\
&&\mbox{} + \big((2\ell-3j-2)^2 - (j-2)^2 -j(\ell-4)(\ell-j-2)\big) \Big].
\end{eqnarray*}
Using the inequality
\begin{equation} \label{quadratic}
|a \phi^2 + b\phi + c| \le 2 a \phi^2 + b^2/4a + |c|,
\end{equation}
valid for all positive $a$, and all real $b,c,\phi$, we have, for all $\ell \ge 8$,
\begin{eqnarray*}
k_3 &\le& \frac{2\ell (\ell-2)}{(\ell-4)^2} \phi^2 + \frac{(4\ell^2 -7\ell j - 8\ell + 12 j)^2}{4\ell (\ell-2) (\ell-4)^2} \\
&& \mbox{} + \frac{ \left|(2\ell-3j-2)^2 -(j-2)^2 -j(\ell-4)(\ell-j-2) \right|}{(\ell-4)^2} \\
&\le& 6\phi^2 + \frac{\left(2(\ell-4)(2\ell-3) - (j-2)(7\ell-12) \right)^2}{4\ell (\ell-2) (\ell-4)^2} \\
&&\mbox{} + \left( \frac{|2\ell-3j-2|}{\ell-4} \right)^2  + \left(\frac{j-2}{\ell-4}\right)^2 + \frac{j(\ell-j-2)}{\ell-4} \\
&\le& 6\phi^2 + 4+4+1 + \frac{j(\ell-j-2)}{\ell-4} \\
&=& 6\phi^2 + 9 + \frac{j(\ell-j-2)}{\ell-4}.
\end{eqnarray*}

Now we combine our bounds to produce a single bound on $k_2k_3$.  The product of $k_2$ with
$\displaystyle 9 + \frac{j(\ell-j-2)}{\ell-4}$ is certainly at most
$\displaystyle 9 + \frac{j(\ell-j-2)}{\ell-4} \le 10 \frac{j(\ell-j-1)}{\ell-4}$,
while the product of $k_2$ with $6\phi^2$ is at most the maximum of
$\displaystyle 6\phi^2 \left(\frac2e\right)^{|\phi|} \le 35 \le 35\frac{j(\ell-j-1)}{\ell-4}$ and
$$
6\phi^2 \exp \left (-\frac{\phi^2}{4} \frac{\ell-4}{(j-2)(\ell-j-2)} \right) \le \frac{24}{e} \frac{(j-2)(\ell-j-2)}{\ell-4}
< 35\frac{j(\ell-j-1)}{\ell-4}.
$$
We can summarise by saying that, provided $\ell \ge 8$, for all values of $\phi$,
$$
k_2k_3 \le 45 \frac{j(\ell-j-1)}{\ell-4}.
$$

Therefore, overall, we have, for $2\le j \le \ell-2$,
\begin{eqnarray*}
4s(\tau-1) |f''_j(s)| &\le& 135 \frac{(\ell-1)(\ell-2)(\ell-3)}{j(j-1)(\ell-j-1)} \sqrt{\frac{\ell-4}{(j-1)(\ell-j-1)}}
\frac{j(\ell-j-1)}{\ell-4} \\
&=& 135 \frac{(\ell-1)(\ell-2)(\ell-3)}{\sqrt{\ell-4}} \frac{1}{(j-1)^{3/2}} \frac{1}{\sqrt{\ell-j-1}} \\
&\le& 560 \frac{(\ell-1)^{5/2}} {j^{3/2}(\ell-j)^{1/2}},
\end{eqnarray*}
provided $\ell \ge 8$.

\smallskip

Consider now the special case $j=\ell-1$, when the bound (\ref{eq:second}) translates to
$\displaystyle 4s (\tau-1) |f''_j(s)| \le 560 (\ell-1)$.  Using (\ref{quadratic}), we have
\begin{eqnarray*}
4s(\tau-1) |f''_{\ell-1}(s)| &=& \frac{(1-v)^{\ell-3}}{v} \left| (\ell-1)^2 v^2 - (\ell-1)v - (1-v)^2 \right| \\
&\le & \frac{(1-v)^{\ell-3}}{v} \left(2(\ell-1)^2 v^2 + \frac54\right) \\
&\le & 2(\ell-1)^2 v (1-v)^{\ell-3} + \frac{2(1-v)^{\ell-3}}{v}.
\end{eqnarray*}

The first term above is maximised at $v=1/(\ell-2)$, so we have
$$
2(\ell-1)^2 v (1-v)^{\ell-3} \le 2 \frac{(\ell-1)^2}{\ell-2} \left( 1- \frac{1}{\ell-2}\right)^{\ell-3} \le 2 (\ell-1).
$$
If $s > (\tau-1)/\ell^2$, then $v > 1/\ell$, and the second term above is at most
$2\ell (1-v)^{\ell-3} \le 3(\ell-1)$, so we do have $4s (\tau-1) |f''_j(s)| \le 5 (\ell-1)$, as desired.

For the range $(\tau-1)/\ell^3 < s \le (\tau-1)/\ell^2$, we have $v > \ell^{-3/2}$.  This gives
$4s(\tau-1) |f''_{\ell-1}(s)| \le 2(\ell-1) + 2 \ell^{3/2}$, and thence $|f''_{\ell-1}(s)| \le \ell^{3/2}/s(\tau-1)$, as claimed.

This completes the proof.
\end{proof}

We are now ready to bound the difference $e_j(s) = a_j(s) - f_j(s)$.
Recall that we have, from comparing the recurrences satisfied by the two systems:
$$
e_j(s-1) = \left( 1 - \frac{\ell-j}{2s} \right) e_j(s) + \frac{\ell-j}{2s} e_{j-1}(s) - \Big[ f_j(s-1) - f_j(s) + f'_j(s) \Big].
$$

For $0\le j \le \ell-1$ and $\tau_0\le s \le \tau-1$, set
$$
C_j = 1 + 140 \sum_{i=1}^j \left( \frac{\ell}{i(\ell-i)} \right)^{3/2}.
$$
We now use induction on $j$ and $\tau-s$ to show that $|e_j(s)| \le C_j \frac{\ell}{\tau-1}$,
for all $0\le j \le \ell-2$ and $\tau_0\le s \le \tau-1$.  (We shall return to the case $j=\ell-1$ afterwards.)

We first check the inequality $\displaystyle |e_0(s)| \le \frac{\ell}{\tau-1}$ for $j=0$.  We have that
$a_0(\tau-1) = 1$, and, for $\tau_0 < s \le \tau-1$,
$$
a_0(s-1) = \left( 1 - \frac{\ell}{2s}\right) a_0(s),
$$
so
$$
a_0(s) = \prod_{w=s+1}^{\tau-1} \left( 1 - \frac{\ell}{2w} \right),
$$
while
$$
f_0(s) = \left( \frac{s}{\tau-1} \right)^{\ell/2} = \prod_{w=s+1}^{\tau-1} \left( 1 - \frac{1}{w} \right)^{\ell/2}.
$$
Thus we have
\begin{eqnarray*}
\frac{f_0(s)}{a_0(s)} = \prod_{w=s+1}^{\tau-1} \frac{ (1-1/w)^{\ell/2}}{(1-\ell/2w)}.
\end{eqnarray*}
Each term in the product is clearly at least~1, so $f_0(s) \ge a_0(s)$ for all $s$.  If $s \le \ell$, then we certainly have
$|e_0(s)| \le f_0(s) \le s/(\tau-1) \le \ell/(\tau-1)$, so we may assume that $s \ge \ell$.  Now we have, for
all $w \ge s \ge 2\ell$,
$$
\frac{(1 - 1/w)^{\ell/2}}{1-\ell/2w} \le \frac{1 - \frac{\ell}{2w} + \frac{\ell^2}{8w^2}}{1- \frac{\ell}{2w}} \le
1+ \frac{1}{1 - 1/2} \frac{\ell^2}{8w^2} = 1 + \frac{\ell^2}{4w^2} \le \exp\left( \frac{\ell^2}{4w^2} \right).
$$

This means that
$$
\frac{f_0(s)}{a_0(s)} \le \exp\left( \sum_{w=s+1}^{\tau-1} \frac{\ell^2}{4 w^2} \right)
\le \exp\left( \frac{\ell^2}{4} \left( \frac{1}{s} - \frac{1}{\tau-1} \right) \right)
= \exp \left( \frac{\ell^2 (\tau-1-s)}{4s(\tau-1)} \right).
$$

Now we write
\begin{eqnarray*}
|e_0(s)| &=& f_0(s)(1-a_0(s)/f_0(s)) \\
&\le& \left( \frac{s}{\tau-1} \right)^{\ell/2} \left( 1 - \exp(-\ell^2(\tau-1-s)/4s(\tau-1) \right)\\
&\le& \left(\frac{s}{\tau-1}\right)^{\ell/2} \frac{\ell^2(\tau-1-s)}{4s(\tau-1)}.
\end{eqnarray*}
This function is maximised at $s=(1-2/\ell)(\tau-1)$, and its value there is equal to
$$
\frac{(1-2/\ell)^{\ell/2 -1} \ell^2 (2/\ell) }{4(\tau-1)} \le \frac{\ell}{2(\tau-1)},
$$
as required for the case $j=0$.

\smallskip

For $j>0$, we have $e_j(\tau-1) = a_j(\tau-1) - f_j(\tau-1) = 0 - 0 = 0$.  Now, for the induction step, suppose that
$0<j\le\ell-2$, $\tau_0 < s\le \tau-1$, and that we have verified our inequality for both $|e_{j-1}(s)|$ and $|e_j(s)|$.
Hence we have
\begin{eqnarray*}
\lefteqn{|e_j(s-1)|}\\
&\le& \left( 1 - \frac{\ell-j}{2s} \right) |e_j(s)| + \frac{\ell-j}{2s} |e_{j-1}(s)| + |f_j(s-1) - f_j(s) + f'_j(s)| \\
&\le& \frac{\ell}{\tau-1} \left[ \left( 1 - \frac{\ell-j}{2s} \right) C_j + \frac{\ell-j}{2s} C_{j-1} \right]
 + |f_j(s-1) - f_j(s) + f'_j(s)|.
\end{eqnarray*}

We now note that $f_j(s-1) - f_j(s) + f'_j(s) = f_j''(w)/2$ for some $w$ in $(s-1,s)$,
and so its absolute value is at most $\displaystyle \frac{70}{s(\tau-1)} \frac{\ell^{5/2}}{j^{3/2}(\ell-j)^{1/2}}$,
by Lemma~\ref{lem:f''},
Therefore
\begin{eqnarray*}
|e_j(s-1)| &\le& \frac{\ell}{\tau-1} \left[ \left( 1 - \frac{\ell-j}{2s} \right) C_j + \frac{\ell-j}{2s} C_{j-1} + \frac{70}{s} \frac{\ell^{3/2}}{j^{3/2}(\ell-j)^{1/2}} \right]\\
&=& \frac{\ell}{\tau-1} \left[ C_j + \frac{\ell-j}{2s} \left( -C_j + C_{j-1} + 140 \frac{\ell^{3/2}}{j^{3/2}(\ell-j)^{3/2}} \right) \right]\\
&=& C_j \frac{\ell}{\tau-1},
\end{eqnarray*}
where the last line is by the definition of the $C_j$'s.

\smallskip

For $j=\ell-1$ and $s-1 > (\tau-1)/\ell^2$, the same calculation still gives us that $|e_{\ell-1}(s-1)| \le C_{\ell-1} \ell/(\tau-1)$, since
the bound~(\ref{eq:second}) is valid for $|f''_{\ell-1}(w)|$ as long as $w > (\tau-1)/\ell^2$.
For values of $s-1$ with $(\tau-1)/\ell^3 < s-1 \le (\tau-1)/\ell^2$, we replace the bound on the second derivative by
$\displaystyle |f''(w)| \le \frac{\ell^{3/2}}{s(\tau-1)}$,
and the same calculation gives
$$
|e_{\ell-1}(s)| \le C_{\ell-2} \frac{\ell}{\tau-1} + \frac{\ell^{3/2}}{\tau-1}.
$$

We now observe that the sum $\displaystyle \sum_{i=1}^{\ell-1} \frac{\ell^{3/2}}{i^{3/2}(\ell-i)^{3/2}}$ is uniformly bounded, being at most
$\displaystyle 2 \sum_{i=1}^{\ell/2} \frac{2^{3/2}}{i^{3/2}} \le 15$.  Therefore we have
$$
|e_j(s)| \le \frac{2200\ell}{\tau-1},
$$
for all $0\le j \le \ell-2$ and $\tau_0\le s\le \tau-1$, and also for $j=\ell-1$ and $(\tau-1)/\ell^2 < s \le \tau-1$.

For $(\tau-1)/\ell^3 < s \le (\tau-1)/\ell^2$, we have
$$
|e_{\ell-1}(s)| \le \frac{2200 \ell}{\tau-1} + \frac{\ell^{3/2}}{\tau-1} \le \frac{800\ell^{3/2}}{\tau-1},
$$
since $\ell \ge 8$.0000

This completes the proof of Lemma~\ref{lem:ejs}, and hence in turn the proofs of Lemma~\ref{PhivR} and Theorem~\ref{thm:main}.

\section{Concentration for $U_t(\ell)$} \label{Usimple}

In this section, we give a very brief sketch of the proof of Theorem~\ref{thm:main-U}.  The proof proceeds on very similar
lines to that of Theorem~\ref{thm:main}.

Recall that $U_t(\ell)$ is the number of vertices of degree {\em at least} $\ell$ at time $t$.  It is easy to show that the
expected value of $U_t(\ell)$ is close to $u_t(\ell) = 2t/\ell(\ell+1)$, uniformly over $t$ and $\ell$.

The difference $F_t(\ell) = U_t(\ell) - u_t(\ell)$ satisfies the matrix equation
$$
F_t = W_{t-1} F_{t-1} + \Delta M_t.
$$
Here $\Delta M_t$ is a vector of martingale differences, and
\begin{eqnarray*}
W_u = \begin{pmatrix}
1 - \frac{1}{2u} & 0 & \cdots & 0\\
\frac{2}{2u} & 1-\frac{2}{2u} &  \cdots & 0\\
\vdots & \ddots & \ddots & 0\\
0 & 0 & \frac{\ell_0-1}{2u} & 1-\frac{\ell_0-1}{2u}
\end{pmatrix}.
\end{eqnarray*}
The indexing of vectors and matrices runs from $\ell =2$ to $\ell=\ell_0$; we do not need to track the component
$\ell=1$ since $U_t(1) =t$ for all $t$.  As for $D_t$, the plan is to apply Theorem~\ref{thm.key} to the martingale
$$
\widetilde{M}^\tau_t = \sum_{s=\tau_0+1}^t C_{s-1} \Delta M_s,
$$
where $C_s = \prod_{u=s+1}^{\tau-1} W_u$, for fixed $\tau > \tau_0$.

A transition of $U_s$ involves an increase of~1 in at most one component $U(\ell)$; in other words $U_{s+1}$ is
obtained from $U_s$ by adding some unit vector $e_\ell$, or leaving the vector unchanged.

For $\ell = 2,\dots, \ell_0$, and $\tau_0 \le t < \tau$, we set
$$
\Phi^\ell_t(Z) = \Phi^{g^\ell}_t(Z) = \sum_{s=\tau_0}^{t-1} \sum_{k=2}^{\ell_0} P_s(U_s,U_s+e_k) C_t(\ell,k)^2.
$$

We set, for each $1\le \ell \le \ell_0$,
$$
S^\ell = 225 \frac{\tau-1}{\ell^2} + 10^{18} \psi^2 \log^{13} (\psi\tau),
$$
and
$$
T_S^\ell = \inf \{ t \ge \tau_0 : \Phi^\ell_t(Z) > S^\ell \}.
$$
An application of Theorem~\ref{thm.key} gives
$$
\P \left( \left( \sup_{\tau_0\le t \le \tau} |\widetilde{M}_t^\tau(\ell)| > 3 \sqrt{\log (\psi\tau) S^\ell} \right)
\wedge (T_S^\ell \ge \tau) \right) \le \frac{2}{\psi^2\tau^2}
$$
for each $\tau \ge \tau_0$ and $2 \le \ell \le \ell_0$.

Now let $\widehat T$ be the infimum of the times $s \ge \tau_0$ such that either there is a vertex of degree at least $\psi \sqrt{s-1}$,
or, for some $k \le \ell_0$,
$$
\left|U_s(k) - \frac{2s}{k(k+1)}\right| >
45 \frac{\sqrt{s \log (\psi s)}}{k} + 4 \times 10^9 \psi \log^7(\psi s).
$$

The next step is to prove the following result, analagous to Lemma~\ref{PhivR}.

\begin{lemma} \label{PhivS}
For all $t\le \widehat{T}$, and all $\ell = 2,\dots, \ell_0$,
$\Phi^\ell_{t-1}(Z) \le S^\ell$.
\end{lemma}

The result implies that $T_S^\ell \ge \widehat{T}$ for all $\ell=2, \dots, \ell_0$.

\smallskip

Now, for each $\tau \ge \tau_0$ and each $\ell \ge 2$, set
$$
\delta_\tau(\ell) = 45 \frac{\sqrt{\tau \log (\psi\tau)}}{\ell} + 3\times 10^9 \psi \log^7 (\psi\tau) .
$$
We have
$$
\P \left( \left(\sup_{\tau_0 \le t \le \tau} |\widetilde{M}^\tau_t(\ell)| > \delta_\tau(\ell) \right) \wedge
(T_S^\ell \ge \tau) \right) \le \frac{2}{\psi^2\tau^2}.
$$
In particular, together with Lemma~\ref{PhivS} and the fact that $\widetilde{M}^\tau_\tau = F_\tau$ for each $\tau \ge \tau_0$, this implies that
$$
\P \left( \left(|F_{\tau}(\ell)| > \delta_\tau(\ell)\right) \wedge (\widehat{T} \ge \tau) \right)
\le \frac{2}{\psi^2 \tau^2}.
$$
We next use this inequality to show that,
with probability at least $3/\psi$, for all $\tau \ge \tau_0$ and $\ell =2, \dots, \ell_0$, either $|F_\tau(\ell)| \le \delta_\tau(\ell)$
or $\tau \ge \widehat{T}$.  Similarly to the proof of Theorem~\ref{thm:main}, this leads to the conclusion that
$\P( \widehat T < \infty ) \le 4/\psi$, which is the desired result.

\section{More complex preferential attachment models}

In this section, we discuss some of the issues we confront when extending this proof to other models of preferential attachment.

A first extension would cover the model which again generates a random tree, where now an arriving vertex chooses an existing vertex $v$ as a neighbour with
probability proportional to $X(v) + \beta$, where $X(v)$ is the degree of vertex~$v$, and $\beta$ is a fixed constant.  For such a model, the expected
degree of a vertex at time $t$ grows as $C t^{1/(2+\beta)}$, and the expected number of vertices of degree $\ell$ at time $t$ behaves as
$C t / \ell^{3+\beta}$.  When attempting to follow the proof in this paper to establish concentration results, the main difficulty is in finding a
suitable analogue of Lemma~\ref{lem:ejs}, giving bounds on the error function playing the role of $e_j(s)$.

Another well-studied variant is to have each arriving vertex select some fixed number $m$ of neighbours (with replacement), instead of just one.  The main
difficulty introduced by this variation is that we have to account for the possibility that some existing vertex has its degree increased by more than one at
each step, and that the recurrence relations do not have such clean forms.

In the full Cooper-Frieze model (see~\cite{cf03}, \cite{c05}), the number of new edges added at each step is a random variable.  Indeed, with some probability,
no new vertex is added, and some edges are added between existing vertices, chosen either uniformly or via preferential attachment.  This means that the
numbers of vertices and edges present at time~$t$ are no longer determined, causing further complications in the application of our method.

We do believe that all of these problems can be overcome, and that our method can be used to analyse general Cooper-Frieze models.  We also hope that the
method will find further applications in the analysis of other random processes.

%\begin{acks}
%
%\end{acks}

\newcommand\AAP{\emph{Adv. Appl. Probab.} }
\newcommand\JAP{\emph{J. Appl. Probab.} }
\newcommand\JAMS{\emph{J. \AMS} }
\newcommand\MAMS{\emph{Memoirs \AMS} }
\newcommand\PAMS{\emph{Proc. \AMS} }
\newcommand\TAMS{\emph{Trans. \AMS} }
\newcommand\AnnMS{\emph{Ann. Math. Statist.} }
\newcommand\AnnPr{\emph{Ann. Probab.} }
\newcommand\CPC{\emph{Combin. Probab. Comput.} }
\newcommand\JMAA{\emph{J. Math. Anal. Appl.} }
\newcommand\RSA{\emph{Random Struct. Algor.} }
\newcommand\SPA{\emph{Stoch. Proc. Appl.} }
\newcommand\DMTCS{\jour{Discr. Math. Theor. Comput. Sci.} }

\newcommand\AMS{Amer. Math. Soc.}
\newcommand\Springer{Springer}
\newcommand\Wiley{Wiley}

\newcommand\vol{\textbf}
\newcommand\jour{\emph}
\newcommand\book{\emph}
\newcommand\inbook{\emph}
\def\no#1#2,{\unskip#2, no. #1,} %(typeset after year)

\newcommand\webcite[1]{\hfil\penalty0\texttt{\def~{\~{}}#1}\hfill\hfill}
\newcommand\webcitesvante{\webcite{http://www.math.uu.se/\~{}svante/papers/}}
\newcommand\arxiv[1]{\webcite{http://arxiv.org/#1}}

\end{document}